 \documentclass[11pt]{article}
 \usepackage[top=1in,bottom=1in,left=1.5in,right=1.5in]{geometry}
 \usepackage{amsmath,amssymb}
  \usepackage{latexsym}
\usepackage{graphicx,color} 

  \newcommand{\C}{\mathbb{C}}
  \newcommand{\F}{\mathbb{F}}
  \newcommand{\N}{\mathbb{N}}
  \renewcommand{\P}{\mathbb{P}}
  \newcommand{\R}{\mathbb{R}}
  \newcommand{\Z}{\mathbf{Z}}

\newcommand{\ii}{\mathbf{i}}

  \newcommand{\e}{\mathbf{e}}
  
  \newcommand{\bi}{\mathbf{i}}
  \newcommand{\bF}{\mathbf{F}}
  
  \newcommand{\g}{\mathbf{g}}
  \newcommand{\gl}{\mathbf{GL}}
  \newcommand{\m}{\mathbf{m}}

  \newcommand{\br}{\mathbf{r}}

  \newcommand{\bt}{\mathbf{t}}

  \newcommand{\U}{\mathbf{U}}
  \newcommand{\uu}{\mathbf{u}}
  \newcommand{\vv}{\mathbf{v}}
  \newcommand{\V}{\mathbf{V}}
  \newcommand{\w}{\mathbf{w}}
  
  \newcommand{\x}{\mathbf{x}}
  
  \newcommand{\y}{\mathbf{y}}
  
  \newcommand{\z}{\mathbf{z}}
  \newcommand{\0}{\mathbf{0}}

  \newcommand{\bomega}{\mbox{\boldmath{$\omega$}}}

  \newcommand{\cO}{\mathcal{O}}

  \newcommand{\cS}{\mathcal{S}}
  \newcommand{\cT}{\mathcal{T}}

  \newcommand{\cX}{\mathcal{X}}

  \newcommand{\rD}{\mathrm{D}}

  \newcommand{\rH}{\mathrm{H}}
  \newcommand{\rO}{\mathrm{O}}

  \newcommand{\Res}{\mathrm{res}}
  \newcommand{\lan}{\langle}
  \newcommand{\ran}{\rangle}
  \newcommand{\an}[1]{\lan#1\ran}
  \def\diag{\mathop{{\rm diag}}\nolimits}
  \newcommand{\hs}{\hspace*{\parindent}}
  \newcommand{\proof}{\hs \textbf{Proof.\ }}

  \newcommand{\trans}{^\top}
  \newcommand{\qed}{\hspace*{\fill} $\Box$\\}

  \newcommand{\dist}{\mathrm{dist}}

  \newcommand{\rS}{\mathrm{S}}

  \newcommand{\rank}{\mathrm{rank\;}}

  \newtheorem{theo}{\bfseries \hs Theorem}
  \newtheorem{defn}[theo]{\bfseries \hs Definition}
  \newtheorem{prop}[theo]{\bfseries \hs Proposition}
  
  \newtheorem{lemma}[theo]{\bfseries \hs Lemma}
  \newtheorem{corol}[theo]{\bfseries \hs Corollary}

  \newtheorem{rem}[theo]{\bfseries \hs Remark}
  
 \def\rig#1{\smash{ \mathop{\longrightarrow}
    \limits^{#1}}}
 \def\sear#1{\searrow
   \rlap{$\vcenter{\hbox{$\scriptstyle#1$}}$}}
 \def\swar#1{\swarrow
   \rlap{$\vcenter{\hbox{$\scriptstyle#1$}}$}}

 \numberwithin{equation}{section} 

 \renewcommand{\span}{\mathrm{span}}

\begin{document}

 \title{The number of singular vector tuples and\\uniqueness of best rank one approximation of tensors}

 \author{
 Shmuel Friedland\footnotemark[1] \and
 Giorgio Ottaviani\footnotemark[2]
 }
 \renewcommand{\thefootnote}{\fnsymbol{footnote}}
 \footnotetext[1]{
 Dept. of Mathematics, Statistics and Computer Science,
 Univ. of Illinois at Chicago, Chicago, Illinois 60607-7045,
 USA, \texttt{friedlan@uic.edu}.
 Supported by NSF grant DMS--1216393.
 }
 \footnotetext[2]{
 Dipartimento di Matematica e Informatica ``Ulisse Dini'', Universit\`a di Firenze,
viale Morgagni 67/A, 50134 Firenze, Italy,
 \texttt{ottavian@math.unifi.it}.
Member of INDAM.
 }

 \renewcommand{\thefootnote}{\arabic{footnote}}
 \date{}
 \maketitle
 \begin{abstract}
 In this paper we discuss the notion of singular vector tuples of a complex valued $d$-mode tensor of dimension
 $m_1\times\ldots\times m_d$.  We show that a generic tensor has a finite number of singular vector tuples, viewed as points
 in {the} corresponding Segre product.  We give the formula for the number of singular vector tuples.
 We show similar results for tensors with partial symmetry.
 We give analogous results for the homogeneous pencil eigenvalue problem for cubic tensors, i.e. $m_1=\ldots=m_d$.
We show uniqueness of best approximations for almost all real tensors in the following cases: rank one approximation; 
 rank one approximation for partially symmetric tensors
(this approximation is also partially symmetric); rank-$(r_1,\ldots,r_d)$ approximation for $d$-mode tensors.
 
 \end{abstract}

 \noindent {\bf 2010 Mathematics Subject Classification.} 14D21, 15A18, 15A69, 65D15, 65H10, 65K05.

 \noindent {\bf Key words.} Singular vector tuples, vector bundles, Chern
 classes, partially symmetric tensors, homogeneous pencil eigenvalue problem for cubic tensors,
 {singular} value decomposition, best rank one approximation, best rank-$(r_1,\ldots,r_d)$ approximation.

 \renewcommand{\thefootnote}{\arabic{footnote}}

 \section{Introduction}\label{sec:intro}
 The object of this paper is to study two closely related topics:
 counting the number of singular vector tuples of complex tensor 
 and the uniqueness of best rank one approximation of {real tensors.   
To state  our results we introduce notation that will be used in this paper.
Let $\F$  be either the field of real or complex numbers, denoted by $\R$ and $\C$ respectively, unless stated otherwise.    
For each $\x\in\F^m\setminus\{\0\}$ we denote
 by $[\x]:=\span(\x)$ the line through the origin spanned by $\x$ in $\F^m$.  Then $\P(\F^{m})$ is the space of all lines
 through the origin in $\F^m$.  We say that $\x\in \F^m,[\y]\in\P(\F^m)$ are generic if there exist subvarietes
$U\subsetneq \F^m, V\subsetneq \P(\F^m)$ such that $\x\in\F^m\setminus U, [\y]\in\P(\F^m)\setminus V$.  A set $S\subset \F^m$ is called closed 
if it is a closed set in the Euclidean topology.   We say that a property $P$ holds almost everywhere in $\R^n$, abbreviated as a.e.,
 if $P$ does not hold on a measurable set $S\subset \R^n$ of zero Lebesgue measure.  Equivalently, we say that almost all
 $\x\in\R^n$ satisfy $P$, abbreviated as a.a..}

For $d\in\N$ denote $[d]:=\{1,\ldots,d\}$.  Let $m_i\ge 2$ be an integer for $i\in [d]$.
 Denote $\m:=(m_1,\ldots,m_d)$.
Let {$\Pi_\F(\m):=\P(\F^{m_1})\times \ldots \times \P(\F^{m_d})$}.   We call $\Pi_\F(\m)$ {the} Segre product.  Set $\Pi(\m):=\Pi_\C(\m)$.
 Denote by $\F^{\m}=\F^{m_1\times\ldots\times m_d}:=\otimes_{i=1}^d \F^{m_i}$ the vector space of $d$-mode tensors
 $\cT=[t_{i_1,\ldots,i_d}], i_j=1,\ldots,m_j, j=1,\ldots,d$ over $\F$.  (We assume that $d\ge 3$ unless stated otherwise.)
 For an integer $p\in [d]$ and for $\x_{j_r}\in \F^{m_{j_r}}, r\in [p] $, we use the notation $\otimes_{j_r,r\in[p]}\x_{j_r}:=\x_{j_1}\otimes\ldots\otimes \x_{j_p}$.
 For a subset $P=\{j_1,\ldots,j_p\}\subseteq\emph{}[d]$ of cardinality $p=|P|$,
 consider a $p$-mode tensor $\cX=[x_{i_{j_1},\ldots,i_{j_p}}]\in\otimes_{j_r, r\in [p]}\F^{m_{j_r}}$, where $j_1<\ldots <j_p$.
 Define 
\[\cT\times \cX:=\sum_{i_{j_r}\in [m_{j_r}], r\in [p]} t_{i_1,\ldots,i_d} x_{i_{j_1},\ldots,i_{j_p}}\]
to be a $(d-p)$-mode tensor obtained 
by contraction on the indices $i_{j_1},\ldots,i_{j_p}$.

To motivate our results let us consider the classical
 case of matrices, i.e. $d=2$ and $A\in \R^{m_1 \times m_2}$.
 We call a pair $(\x_1,\x_2)\in (\R^{m_1}\setminus\{\0\})\times (\R^{m_2}\setminus\{\0\})$ a \emph{singular vector pair} if
 \begin{equation}\label{defsingpr}
 A\x_2=\lambda_1\x_1,\quad A\trans \x_1=\lambda_2 \x_2,
 \end{equation}
 for some $\lambda_1,\lambda_2\in\R$.
 For $\x\in\R^m$ let $\|\x\|:=\sqrt{\x\trans \x}$ be the Euclidean norm on $\R^m$.
 Choosing $\x_1,\x_2$ to be of Euclidean length one we deduce that
 $\lambda_1=\lambda_2$, where $|\lambda_1|$ is equal to some singular value of $A$.
 It is natural to identify all singular vector pairs of the form $(a_1\x_1,a_2\x_2)$, where $a_1a_2\ne 0$
 as the class of singular vector pair.    Thus $([\x_1],[\x_2])\in \P(\R^{m_1})\times \P(\R^{m_2})$ is called a singular vector pair of $A$.

For a generic $A$, i.e. $A$ of the maximal rank $r=\min(m_1,m_2)$ and $r$ distinct
 positive singular values, $A$ has exactly $r$ distinct singular vector pairs.  Furthermore, under these conditions
 $A$ has a unique best rank one approximation in the Frobenius norm given by the singular vector pair corresponding to
 the maximal singular value \cite{GolV96}.

 Assume now that $m=m_1=m_2$ and $A$ is a real symmetric matrix.
 Then the singular values of $A$ are the absolute values
 of the eigenvalues of $A$.  Furthermore, if all the absolute values of the eigenvalues of
 $A$ are pairwise distinct then $A$ has a unique best rank one approximation, which is symmetric.
 Hence for any real symmetric matrix $A$ there exists a best rank one approximation which is symmetric.

  In this paper we derive similar results for tensors.
 Let $\cT\in \F^{\m}$.
 We first define the notion of a singular vector tuple $(\x_1,\ldots,\x_d)\in (\F^{m_1}\setminus\{\0\})\times\ldots\times (\F^{m_d}\setminus\{\0\})$
 {\cite{Lim05}}:
 \begin{equation}\label{defsingtpl}
 \cT\times \otimes_{j\in [d]\setminus \{i\}} \x_j= \lambda_i \x_i, \quad i=1,\ldots,d.
 \end{equation}
 As for matrices we identify all singular vector tuples of the form $(a_1\x_1,\ldots,a_d\x_d)$, $a_1\ldots a_d\ne 0$
 as one class of singular vector tuple in $([\x_1],\ldots,[\x_d])\in\Pi_\F(\m)$.  (Note that for $d=2$ and $\F=\C$
 our notion of singular vector pair differs from the classical notion of singular vectors for complex-valued matrices, see \S3.)

Let $([\x_1],\ldots,[\x_d])\in \Pi(\m)$ be a singular vector tuple of $\cT\in\C^{\m}$.   This tuple corresponds to a zero (nonzero) singular value
if $\prod_{i\in[d]}\lambda_i=0\; (\ne 0)$.  This tuple
 is called a simple singular vector tuple, (or just simple), if the corresponding global section corresponding to $\cT$ has a simple zero at $([\x_1],\ldots,[\x_d])$,  
see Lemma \ref{singularsection} in \S 3.

 {Our first major result is:
\begin{theo}\label{numbsingtup}  
 Let $\cT\in \C^{\m}$ be generic.  Then  $\cT$ has exactly $c(\m)$ simple singular vector tuples which correspond to nonzero singular values.
 Furthermore, $\cT$ does not have a zero singular value.
 In particular, a generic real-valued tensor $\cT\in \R^\m$ has at most $c(\m)$ real singular vector tuples corresponding to nonzero singular values,
 and all of them are simple.  The integer $c(\m)$ is the coefficient of the monomial $\prod_{i=1}^ d t_i^{m_i-1}$ in the polynomial
\begin{equation}\label{genpolcm1}
 \prod_{i\in [d]}\frac{\hat t_i^{m_i} -t_i^{m_i}}{\hat t_i - t_i}, \quad \hat t_i=\sum_{j\in [d]\setminus\{i\}} t_j,\; i\in [d].
 \end{equation}
 \end{theo}}

 At the end of \S \ref{sec:numsingtpl} we list the first values of $c(\m)$ for $d=3$.
 We generalize the above results to the class of tensors with given partial symmetry.

 We now consider the cubic case where $m_1=\ldots=m_d=m$.  For an integer $m\ge 2$ let $m^{\times d}:=(\underbrace{m,\ldots,m}_d)$.  
Then $\cT\in\F^{m^{\times d}}$ is called $d$-cube, or simply a cube tensor.  For a vector $\x\in\C^m$ let $\otimes^k\x:=\underbrace{\x\otimes\ldots\otimes\x}_{k}$.
 Assume that $\cT,\cS\in \C^{m^{\times d}}$.  Then the homogeneous pencil eigenvalue problem is to find all vectors $\x$ and scalars
$\lambda$ satisfying
 $\cT\times \otimes^{d-1}\x=\lambda \cS\times \otimes^{d-1}\x$.  The contraction here is with respect to the last $d-1$ indices of $\cT,\cS$ respectively.
 We assume without loss of generality that $\cT=[t_{i_1,\ldots,i_d}],\cS=[s_{i_1,\ldots,i_d}]$ are symmetric with respect to the indices $i_2,\ldots,i_d$.
 $\cS$ is called nonsingular if the system $\cS\times \otimes^{d-1}\x=\0$ implies that $\x=0$.  Assume that $\cS$ is nonsingular and fixed.
 Then $\cT$ has exactly $m(d-1)^{m-1}$ eigenvalues counted with their multiplicities.  $\cT$ has $m(d-1)^{m-1}$ distinct eigenvectors in
 $\P(\C^m)$ for a generic $\cT$.  {See \cite{Qi05} for the case $\cS$ is the identity tensor.}

 View $\R^{m_1\times \ldots m_d}$ as
 an inner product space, where for two d-mode tensors $\cT,\cS \in\R^{m_1\times\ldots\times m_d}$ we let
 $\an{\cT,\cS}:= \cT\times \cS$.  Then the Hilbert-Schmidt norm is defined $\|\cT\|:=\sqrt{\an{\cT,\cT}}$.
 ({Recall that for $d=2$ (matrices) the Hilbert-Schmidt norm is called the Frobenius norm.})
 {A} best rank one approximation is a solution to the minimal problem
 \begin{equation}\label{brnk1ap}
 \min_{\x_i\in\R^{m_i},i\in[d]} \|\cT-\otimes_{i\in[d]}\x_i\|=\|\cT-\otimes_{i\in[d]}\uu_i\|.
 \end{equation}
 $\otimes_{i\in[d]}\uu_i$ is called a best rank one approximation of $\cT$.  
 Our second major result is:
\begin{theo}\label{fntbr1ap+s} 
\begin{enumerate}  
\item 
 For almost all $\cT\in \R^{\m}$ a best rank one approximation is unique.
\item Let $\rS^d(\R^m)\subset \R^{m^{\times d}}$ be the space of $d$-mode symmetric tensors.
For almost all $\cS\in \rS^d(\R^m)$ a best rank one approximation
of $\cS$ is unique and symmetric.  In particular, for each $\cS\in \rS^d(\R^m)$ there exists a best rank one approximation which is symmetric.
\end{enumerate}
 \end{theo}

The last statement of part \emph{2} of this theorem was demonstrated by the first named author in \cite{Fri11}. 
Actually, this result is equivalent to Banach's theorem \cite{Ban38}.  
 See \cite{ZLQ} for another proof of Banach's theorem.
In Theorem \ref{numbsingtuppsym} we generalize part \emph{2} of Theorem \ref{fntbr1ap+s} to the class of tensors with given partial symmetry.

Let $\br=(r_1,\ldots,r_d)$, where $r_i\in [m_i]$ for $i\in [d]$.   In the last section of this paper we study a best rank-$\br$ approximation for a real
 $d$-mode tensor \cite{LMV00}.   We show that for almost all tensors a best rank-$\br$ approximation is unique.

 We now describe briefly the contents of our paper. In \S2 we give layman's introduction to some basic notions of vector bundles over compact complex manifolds 
and Chern classes of certain bundles over the Segre product needed for this paper.
 We hope that this introduction will make our paper accessible to a wider audience.
 \S3 discusses the first main contribution of this paper.  Namely, the number of singular vector tuples of a generic complex tensor is finite
 and is equal to $c(\m)$.  We give a closed formula for $c(\m)$,
as in (\ref{genpolcm1}).
 \S4 generalizes these results to partially symmetric tensors.  In particular we reproduce the result of Cartwright and Sturmfels for
 symmetric tensors \cite{CS}.  In \S5 we discuss a homogeneous pencil eigenvalue problem.
 In \S6 we give certain conditions on a general best approximation problem in $\R^n$, which are probably well known to the experts.
{In \S7 we} give uniqueness results on best rank one approximation of partially symmetric tensors.
 In \S8 we discuss a best rank-$\br$ approximation.

We thank J. Draisma, who pointed out the importance to 
distinguish between isotropic and not isotropic vectors,
as we do in \S 3.

 \section{Vector bundles over compact complex manifolds}\label{sec:vecbun}
 In this section we recall some basic results on complex manifolds and holomorphic tangents bundles that we use in this paper.
 Our object {is} to give the simplest possible intuitive description of basic results in algebraic geometry needed in this paper,
 sometimes compromising the rigor.  An interested reader can consult for more details with \cite{GH78} for general facts about 
complex manifolds and complex vector bundles, and for a simple axiomatic exposition on complex vector bundles with \cite{Kob}.
 For {Bertini-type} theorem we refer to Fulton \cite{Ful} and Hartshorne \cite{Har}.
 \subsection{Complex compact manifolds}\label{subsec:compman}
 Let $M$ be a compact complex manifold of dimension $n$.  Thus there exists a finite open cover $\{U_i\}, i\in [N]$ with coordinate homeomorphism
 $\phi_i: U_i\to \C^n$ such that $\phi_i\circ\phi_j^{-1}$ is holomorphic {on} $\phi_j(U_i\cap U_j)$ for all $i,j$.

 As an example consider the $m-1$ dimensional complex projective space $\P(\C^m)$, which is the set of all complex lines in $\C^{m}$ through the origin.  
Any point in $\P(\C^{m})$ is represented by a one dimensional subspace spanned by the vector $\x=(x_1,\ldots,x_{m})\trans\in\C^{m}\setminus\{\0\}$.  
The standard open cover of $\P(\C^m)$ consists of $m$ open covers
 $U_1,\ldots,U_{m}$, where $U_i$ corresponds to the lines spanned by $\x$ with $x_i\ne 0$.  The homeomorphism $\phi_i$ is given by
 $\phi_i(\x)=(\frac{x_1}{x_i},\ldots, \frac{x_{i-1}}{x_i},\frac{x_{i+1}}{x_i},\ldots,\frac{x_{m}}{x_i})\trans$.  So each $U_i$
 is homeomorphic to $\C^{m-1}$.

 Let $M$ be an $n$-dimensional compact complex manifold as above.
 For $\zeta\in U_i$, the coordinates of the vector $\phi_i(\zeta)=\z=(z_1,\ldots,z_n)\trans$ are called the local coordinates of $\zeta$.
 Since $\C^n\equiv\R^{2n}$, $M$ is a real manifold of real dimension $2n$.
 Let $z_j=x_j+\ii y_j, \bar z_j=x_j-\ii y_j,j\in [n]$, where $\ii=\sqrt{-1}$.  For simplicity of notation we let $\uu=(u_1,\ldots,u_{2n})=
 (x_1,y_1,\ldots,x_n,y_n)$ be the real local coordinates on $U_i$.
 Any function $f:U_i\to \C$ in the local coordinates is viewed as $f(\uu)=g(\uu)+\ii h(\uu)$, where $h,g: U_i\to \R$.
 Thus $df=\sum_{j\in [2n]} \frac{\partial f}{\partial u_j} du_j$.
 For a positive integer  $p$, a (differential) $p$-form $\omega$ on $U_i$ is given in the local coordinates as follows
 \[\omega=\sum_{1\le i_1<\ldots<i_p\le 2n} f_{i_1,\ldots,i_p}(\uu)du_{i_1}\wedge\ldots
 \wedge du_{i_p}.\]
 {($ f_{i_1,\ldots,i_p}(\uu)$ are differentiable functions in local coordinates $\uu$ for $1\le i_1<\ldots<i_p\le 2n$.)}
 Recall that the wedge product of two differential is anti commutative. i.e. $du_k\wedge du_l=-du_l\wedge du_k$.  Then
 \[d\omega=\sum_{1\le i_1<\ldots<i_p\le 2n}  (df_{i_1,\ldots,i_p})\wedge du_{i_1}\wedge\ldots
 \wedge du_{i_p}.\]
 (Recall that a differential $0$-form is a function.)
 Note that for $p>2n$ any differential $p$-form is a zero form.
 A straightforward calculation shows that $d(d\omega)=0$.   $\omega$ is a $p$-form on $M$ if its restriction to each $U_i$ is an 
$p$-form, and the restrictions of these two forms on $U_i\cap U_j$ are obtained one from the other one 
by the change of coordinates $\phi_i\circ \phi_j^{-1}$.
 $\omega$ is called closed if $d\omega=0$, and $d\omega$ is called an exact form.  The space of closed $p$-forms modulo exact $p$-forms
 is a finite dimensional vector space over $\C$, which is denoted by $\rH^p(M)$.  Each element of $\rH^p(M)$ is represented by
 a closed $p$-form, and the difference between two representatives is an exact form.  Since the product of two forms is also a form,
 it follows that the space of all closed forms modulo exact forms {is} a finite dimensional algebra, where the identity $1$ corresponds to the constant function
 with value $1$ on $M$.
 \subsection{Holomorphic vector bundles}\label{subsec:holvec}
 A holomorphic vector bundle $E$ on $M$ of rank $k$, where $k$ is a nonnegative integer, is a complex manifold of dimension $n+k$, 
which can be simply described as follows.  There exists a finite open cover $\{U_i\},i\in [N]$ of $M$ with the properties as above satisfying 
the following additional conditions.  At each $\zeta \in U_i$ we are given $k$-dimensional vector space $E_\zeta$, called a fiber of $E$ over 
$\zeta$, which all can be identified with a fixed vector space $\V_i$, having a basis $[\e_{1,i},\ldots,\e_{k,i}]$.
 For $\zeta \in U_i\cap U_j, i\ne j$ the transition matrix from $[\e_{1,i},\ldots,\e_{k,i}]$ to $[\e_{1,j},\ldots,\e_{k,j}]$
 is given by an $k\times k$ invertible matrix $g_{U_j U_i}(\zeta)$.  So $[\e_{1,i},\ldots,\e_{k,i}]=[\e_{1,j},\ldots,\e_{k,j}]
 g_{U_j U_i}(\zeta)$.
 Each entry of $g_{U_j U_i}(\zeta)$ is a holomorphic function in the local coordinates of $U_j$.  We have the following relations
 \[g_{U_i U_j}(\zeta) g_{U_j U_i}(\zeta)= g_{U_i U_j}(\eta) g_{U_j U_p}(\eta) g_{U_p U_i}(\eta)=I_k \textrm{ for }
 \zeta\in U_i\cap U_j, \eta\in U_i\cap U_j\cap U_p.\]
 ($I_k$ is an identity matrix of order $k$.)

 For $k=0$, $E$ is called a zero bundle.
 $E$ is called a line bundle if $k=1$.  $E$ is called a trivial bundle if there exists a finite open cover such that each $g_{U_i U_j}(\zeta)$
 is an identity matrix.  
 A vector bundle $F$ on $M$ is called a subbundle of $E$ if  $F$ is a submanifold of $E$ such that $F_\zeta$ is a subspace
 of $E_\zeta$ for each $\zeta\in M$.  Assume that $F$ is a subbundle of $E$.  Then $G:=E/F$ is the quotient bundle of $E$ and $F$, where
 $G_\zeta$ is the quotient vector space $E_\zeta/F_\zeta$.
 Let $E_1,E_2$ be two vectors bundles on $M$.  We can create the following new bundles on $M$: $E:=E_1\oplus E_2, F:=E_1\otimes E_2, 
{H:=\textrm{Hom} (E_1,E_2)}$. 
 Here $E_\zeta=E_{1,\zeta}\oplus E_{2,\zeta}, F_\zeta=E_{1,\zeta}\otimes E_{2,\zeta}$ and $H_\zeta$ consists of all linear transformations from $E_{1,\zeta}$ to
 $E_{2,\zeta}$.   {In particular, the vector bundle $\textrm{Hom} (E_1,E_2)$, where $E_2$ is the one dimensional trivial bundle
is called the dual bundle of $E_1$ and is denoted by $E_1^\vee$. 
Recall that $\textrm{Hom} (E_1,E_2)$ is isomorphic to $E_2\otimes E_1^\vee$. 
For a given vector bundle $E$ on $M$ we can define the bundle $F:=\otimes^d E$.  Here $F_\zeta=\otimes^d E_\zeta$ is a fiber of $d$-mode tensors.

 {Let $M,M'$ be compact complex manifolds and assume that $f:M'\to M$ is holomorphic.  Assume that $\pi:E\to M$ is holomorphic vector bundle.
 Then one can pullback $E$ to obtain a bundle $\pi':E'\to M'$ where $E'=f^*E$.  
 
Given a manifold $M_i$ with a vector bundle
 $E_i$ for $i=1,2$ we can define the bundle $F:=E_1\oplus E_2, G:=E_1\otimes E_2$ on $M:= M_1\times M_2$ by the equality
 \[F_{(\zeta_1,\zeta_2)}=E_{1,\zeta_1}\oplus E_{2,\zeta_2}, G_{(\zeta_1,\zeta_2)}=E_{1,\zeta_1}\otimes E_{2,\zeta_2}.\]
 A special case for $F$ occurs when one of the factors $E_i$ is a zero bundle, say $E_2=0$.  Then $E_1\oplus 0$ is the pullback of the bundle 
$E_1$ on $M_1$ obtained by using the projection $\pi_1: M_1\times M_2$ and is denoted as the bundle $\pi_1^* E_1$ on $M_1\times M_2$.
 Thus $E_1\oplus E_2$ is the bundle $\pi_1^* E_1\oplus \pi_2^* E_2$ on $M_1\times M_2$.  Similarly $E_1\otimes E_2$ is the bundle
 $\pi_1^* E_1\otimes \pi_2^* E_2$.

 We now discuss a basic example used in this paper.  Consider the trivial bundle $F(m)$ on $\P(\C^{m})$ of rank $m$.
 So $F(m)_\zeta=\C^m$.
 The tautological line bundle $T(m)$ on $\P(\C^{m})$, customarily denoted by $\cO(-1)$, is given by $T(m)_{[\x]}=\span(\x)\subset\C^m$.  
So $T(m)$ is a subbundle of $F(m)$.  Denote by $Q(m)$ the quotient bundle
 $F(m)/T(m)$.  So $\rank Q(m)=m-1$.  We have an exact sequence of the following bundles on $\P(\C^{m})$
 \begin{equation}\label{exseqtautb}
 0\rightarrow T(m) \rightarrow F(m)\rightarrow Q(m)\rightarrow 0.
 \end{equation}
 The dual of the bundle of $T(m)$, also called the hyperplane line bundle, denoted here by $H(m)$.   ($H(m)$ is customarily denoted by $\cO(1)$
 in the algebraic geometry literature.)
 \subsection{Chern polynomials}\label{subsec:chern}
 We now return to a holomorphic vector bundle $E$ on a compact complex manifold $M$.  
{The seminal work of Chern \cite{Che46} associates with each $\pi:E\to M$ the Chern class $c_j(E)$ for each $j\in[\dim M]$.
One can view $c_j(E)$ as an element in $\rH^{2j}(M)$. The Chern classes needed in this paper can be determined by the following well known rules
\cite{Kob}.}
 
One associate with $E$ the Chern polynomial
 $C(t,E)=1+\sum_{j=1}^{\rank E} c_j(E)t^j$.   Note that $c_j(E)=0$ for $j>\dim M$.
The total Chern class $c(E)$ is $C(1,E)=\sum_{j=0}^{\infty} c_j(E)$.
 {Consider the formal factorization $C(t,E)=\prod_{j=1}^{\rank E}(1+\xi_j(E)t)$. 
Then the Chern character $ch(E)$ of $E$ is defined as $\sum_{j=1}^{\rank E}e^{\xi_j(E)}$.}

 $C(t,E)=1$ if $E$ is a trivial bundle.   The Chern polynomial of the dual bundle is given by $C(t,E^\vee)=C(-t,E)$.
 Given an exact sequence of bundles
$$0\to E\to F\to G\to 0,$$
we have the identity \begin{equation}\label{whitney}C(t,F)=C(t,E)C(t,G),
\end{equation} which is is equivalent to $c(F)=c(E)c(G)$.

 The product formula is the identity $ch(E_1\otimes E_2)=ch(E_1)ch(E_2)$.  Let $f:M'\to M$.  Then $c_j(f^*E)$,
viewed as a differential form in $\rH^{2j}(M')$, is obtained by pullback of the differential form $c_j(E)$.  
 In particular, for the pullback bundle $\pi_1^* E_1$ described above, we have the equality $c_j(\pi_1^* E_1)= c_j(E_1)$,
 when we use the local coordinates $\zeta=(\zeta_1,\zeta_2)$ on $M_1\times M_2$.

Assume that $\rank E=\dim M=n$.  Then $c_n(E)=\nu(E)\omega$, where $\omega\in\rH^{2n}(M)$ is the volume form on $M$ such that
$\omega$ is a generator of $\rH^{2n}(M,\Z)$.  Then $\nu(E)$ is an integer, which is called the top Chern number of $E$.
 
 Denote by $s_m$ the first Chern class of $H(m)$, which belongs to $\rH^2(\P(\C^m))$.  Then $s_m^k$
 represents the differential form  $ \wedge^k s_m\in\rH^{2k}(\P(\C^m))$.  Observe that $s_m^m=0$.
 Moreover the algebra of all closed forms modulo the exact forms on $\P(\C^m)$ is $\C[s_m]/(s_m^m)$,
 i.e. all polynomials in the variable $s_m$ modulo the relation $s_m^m=0$.
 So  $C(t,H(m))=1+ s_mt$ and  $C(t,T(m))=1- s_mt$.
 The exact sequence \eqref{exseqtautb} and the  formula (\ref{whitney}) imply that
 {\[1=C(t,F(m))=C(t,T(m))C(t,Q(m))=(1-s_mt)C(t,Q(m)).\]}  
Therefore
\begin{equation}\label{CtQmfor}
 C(t,Q(m))=\frac{1}{1- s_m t}=1+\sum_{j=1}^{m-1}  s_m^j t^j.
\end{equation}
}

 \subsection{Certain bundles on Segre product}\label{subsec:segre}
 Let $m_1,\ldots,m_d\ge 2$ be given integers with $d>1$. {Denote 

\noindent
$\m_i=(m_1,\ldots,m_{i-1},m_{i+1},\ldots,m_d)$ for $i\in[d]$.  
Consider the Segre product  $\Pi(\m):=\P(\C^{m_1})\times\ldots\times\P(\C^{m_d})$ and $\Pi(\m_i)$ for $i\in[d]$.  
Let $\pi_i:\Pi(\m)\to \P(\C^{m_i})$ and $\tau_i:\Pi(\m)\to \Pi(\m_i)$ be the projections
 on the $i$-th component and its complement respectively.  
Then $\pi_i^* H(m_i),\pi_i^* Q(m_i),\pi_i^*F(m_i)$ are the pullback of the bundles $H(m_i),Q(m_i),F(m_i)$ on $\P(\C^{m_i})$
 to $\Pi(\m)$ respectively.

Consider the map $\iota_\m: \Pi(\m)\to \P(\C^\m)$ given by $\iota_\m([\x_1],\ldots,[\x_d])=[\otimes_{i\in [d]}\x_i]$.
It is straightforward to show that $\iota$ is $1-1$.  Then $\Sigma(\m):=\iota_\m(\Pi(\m)\subset \P(\C^\m)$ is the Segre variety.
Let $T(\m)$ be tautological line bundle on $\P(\C^\m)$.   The identity $\span(\otimes_{j\in[d]}\x_j)=\otimes_{j\in [d]}\span(\x_j)$
implies that the line bundle $\iota^*T(\m)$ is isomorphic to $\otimes_{j\in[d]}\pi_j^*T(m_j)$.  Hence the dual bundles $\iota^*H(\m)$ and 
 $\otimes_{j\in[d]}\pi_j^*H(m_j)$ are isomorphic.  
Consider next the bundle $\hat T(\m_i)$ on $\Pi(\m)$, which is 
\begin{equation}\label{defhattmi}
\hat T(\m_i):=\otimes_{j\in[d]\setminus\{i\}}\pi_j^*T(m_j).
\end{equation}
Hence the dual bundle 
$\hat T(\m_i)^\vee$ is isomorphic to $\otimes_{j\in[d]\setminus\{i\}} \pi_j^* H(m_j)$.
In particular,
\begin{equation}\label{linebdlid}
c_1(\hat T(\m_i)^\vee)=c_1(\otimes_{j\in[d]\setminus\{i\}}\pi_j^*H(m_j)).
\end{equation}

 Define the following vector bundles on $\Pi(\m)$
 \begin{eqnarray}\notag
 &&R(i,\m)=\textrm{Hom}(\hat T(\m_i), \pi_i^*Q(m_i)),\quad R(i,\m)'=\textrm{Hom}(\hat T(\m), \pi_i^*F \;(m_i)),\\
 \label{defRm}\\
 &&R(\m)=\oplus_{i\in [d]} R(i,\m), \quad R_i(\m)':=(\oplus_{j\in[d]\setminus\{i\}} R(j,\m))\oplus R(i,\m)'.
\notag
 \end{eqnarray}

Observe that
\begin{eqnarray}\notag
\rank R(i,\m)= \rank R(i,\m)'-1=m_i-1, \\
\label{rankeqRim}\\
\rank R(\m)=\rank R_i(\m)'-1=\dim \Pi(\m).
\notag
\end{eqnarray}

Since $\textrm{Hom}(E_1,E_2)\sim E_2\otimes E_1^\vee$ we obtain the following relations
\begin{equation}
C(t,R(i,\m))=C(t, \pi_i^*Q(m_i)\otimes(\hat T(\m)^\vee)=C(t,\pi_i^*Q(m_i)\otimes(\otimes_{j\in[d]\setminus\{i\}}\pi_j^*H(m_j)).
 \label{CtRmfor}
\end{equation}

 The formula (\ref{whitney}) yields
\begin{equation}
\label{chernRm}
C(t,R(\m))=\prod_{i\in[d]} C(t,R(i,\m)).
\end{equation}

 Denote $t_i=c_1(\pi_i^* H(m_i))$. The cohomology ring $H^*(\Pi(\m))$
 is generated by $t_1,\ldots, t_d$ with the relations $t_i^{m_i}=0$, that is
 $H^*(\Pi(\m))\simeq \C[t_1,\ldots, t_d]/(t_1^{m_1},\ldots,t_d^{m_d})$
 and in the following we interpret $t_i$ just as variables.
 Correspondingly, the $k$-th Chern class $c_k(E)$ is equal to $p_k(t_1,\ldots,t_d)$ for some homogeneous polynomial $p_k$
of degree $k$ for $k=1,\ldots,\dim \Pi(\m)$.  (Recall that $c_0(E)=1$ and $c_k(E)=0$ for $k>\dim\Pi(\m).$)

In what follows we need to compute the top Chern class of $R(\m)$.
Since $\rank R(\m)=\dim \Pi(\m)$, and $\Pi(\m)$ is a manifold, 
it follows that the top Chern class of $R(\m)$ is of the form 
\begin{equation}\label{topcherncls}
c(\m)\prod_{i\in[d]} t_i^{m_i-1},
\end{equation}
where $c(\m)$ is an integer.  So $c(\m)=\nu(R(\m))$ is the top Chern number of $R(\m)$.
 \begin{lemma}\label{cmfor}  Let $R(i,\m)$ and $R(\m)$ be the vector bundles on the Segre product  $\Pi(\m)$ given by \eqref{defRm}. 
Then the total Chern classes of these vector bundles are given as follows.
 \begin{eqnarray}\label{chernRim}
 &&c(R(i,\m))=\sum_{j=0}^{m_i-1} (1+\hat t_i)^{m_i-1-j} t_i^j,\quad \hat t_i:=\sum_{k\in [d]\setminus\{i\}} t_k,\\
\label{genpolcm}
 &&c(R(\m))=\prod_{i\in[d]}(\sum_{j=0}^{m_i-1} (1+\hat t_i)^{m_i-1-j} t_i^j).
 \end{eqnarray} 
 The top Chern number 
 of $R(\m)$, $c(\m)$, is the coefficient of the monomial $\prod_{i\in[d]} t_i^{m_i-1}$ in the polynomial 
$\prod_{i\in[d]}\frac{\hat t_i^{m_i} -t_i^{m_i}}{\hat t_i - t_i}$,
(In this formula of $c(\m)$ we do not assume  the identities $t_i^{m_i}=0$ for $i\in[d]$.)
 \end{lemma}
 \proof
 Let $\zeta_i:=e^{\frac{2\pi\bi}{m_i}}$ be the primitive $m_i-th$ root of unity.
Then   
\begin{equation}\label{prodformzetx}
\prod_{k=0}^{m_i-1}(1-\zeta_i^kx)=1-x^{m_i},\;\sum_{k=0}^{m_i-1} x^k= \frac{1-x^{m_i}}{1-x}=\prod_{k\in [m_i-1]}(1-\zeta_i^k x).
\end{equation}
 The second equality of \eqref{prodformzetx} and \eqref{CtQmfor} yield that  
$$C(t,\pi_i^*Q(m_i))=\prod_{k\in[m_i-1]}(1-\zeta_i^k  t_i t).$$
  Hence
$ch(\pi_i^*Q(m_i))=\sum_{k\in[m_i-1]} e^{-\zeta_i^k t_i}$.  Clearly, $ch(H(m_j))=e^{t_j}$.  
The product formula for Chern characters yields:
\begin{eqnarray*}
&&ch(\otimes_{j\in[d]\setminus\{i\}} \pi_j H(m_j))=e^{\sum_{j\in[d]\setminus\{i\}} t_j}=e^{\hat t_i},\\ 
&&ch(\pi_i^*Q(m_i)\otimes(\otimes_{j\in[d]\setminus\{i\}}\pi_j^*H(m_j)))=ch(\pi_i^*Q(m_i))ch(\otimes_{j\in[d]\setminus\{i\}}\pi_j^*H(m_j))\\
&&=\sum_{k\in[m_i-1]} e^{\hat t_i-\zeta_i^k t_i}.
\end{eqnarray*}
Hence
\begin{eqnarray*}
&&C(t,R(i,\m))=\prod_{k\in[m_i-1]} (1+(\hat t_i-\zeta_i^kt_i)t)=\frac{1}{1+(\hat t_i -t_i)t} \prod_{k=0}^{m_i-1}(1+(\hat t_i-\zeta_i^kt_i)t),\\
&&c(R(i,\m))=C(1,R(i,\m))=\frac{1}{1+\hat t_i -t_i} \prod_{k=0}^{m_i-1}(1+\hat t_i-\zeta_i^kt_i)=\\
&&\frac{1}{1+\hat t_i -t_i} (1+\hat t_i)^{m_i}\prod_{k=0}^{m_i-1}(1-\zeta_i^k x)=\frac{1}{1+\hat t_i -t_i} (1+\hat t_i)^{m_i}(1-x^{m_i}) ,
\end{eqnarray*}
where $x=\frac{t_i}{1+\hat t_i}$.  As $t_i^{m_i}=0$ we deduce
 \begin{eqnarray*}
&&c(R(i,\m))=\frac{(1+\hat t_i)^{m_i}}{1-t_i+\hat t_i}=\frac{(1+\hat t_i)^{m_i-1}}{1-x}=(1+\hat t_i)^{m_i-1}\sum_{p=0}^{\infty}x^p=\\
&&(1+\hat t_i)^{m_-1}\sum_{p=0}^{m_i-1}x^p=\sum_{j=0}^{m_i-1} (1+\hat t_i)^{m_i-1-j} t_i^j.
\end{eqnarray*}

This establishes \eqref{chernRim}.   \eqref{genpolcm} follows from the 
formula (\ref{whitney}).
Note that the degree of the polynomial in $\bt:=(t_1,\ldots,t_d)$ appearing in the right-hand side of \eqref{chernRim} is $m_i-1$.
The polynomial
$\sum_{j=0}^{m_i-1} \hat t_i^{m_i-1-j} t_i^j=\frac{\hat t_i^{m_i}-t_i^{m_i}}{\hat t_i-t_i}$ is the homogeneous polynomial of degree $m_i-1$
appearing in  the right-hand side of \eqref{chernRim}.  Hence the homogeneous polynomial of degree $\dim \Pi(\m)$ of the right-hand side of 
\eqref{genpolcm} is $\prod_{i\in[d]} \frac{\hat t_i^{m_i}-t_i^{m_i}}{\hat t_i-t_i}$.  Assuming the relations $t_i^{m_i}=0, i\in [d]$ we obtain
that this polynomial is $c(\m)\prod_{i\in[d]} t_i^{m_i-1}$.  This is equivalent to the statement that $c(\m)$ is the coefficient of $)\prod_{i\in[d]} t_i^{m_i-1}$
in the polynomial $\prod_{i\in[d]} \frac{\hat t_i^{m_i}-t_i^{m_i}}{\hat t_i-t_i}$, where we do not assume the relations  $t_i^{m_i}=0, i\in [d]$. \qed

 \subsection{{Bertini-type theorems}}\label{subsec:bertini}
 Let $M$ be a compact complex manifold and $E$ a holomorphic bundle on $M$.
 A holomorphic section $\sigma$ of $E$ on an open set $U\subset E$ is a holomorphic map $\sigma: U\to E$,
 where $E$ is viewed as a complex manifold.  Specifically, let $U_i,i\in [N]$ be the finite cover of $M$
 such that the bundle $E$ restricted to $U_i$ is $U_i\times \C^k$ with the standard basis $[\e_{1,i},\ldots,\e_{k,i}]$,
 as in \S\ref{subsec:holvec}.  Then $\sigma(\zeta)=\sum_{j=1}^k \sigma_{j,i}(\zeta)\e_{j,i}$ for $\zeta\in U\cap U_i$.
 where $\sigma_{j,i}(\zeta), j\in [k]$ are analytic on $U\cap U_i$.
 $\sigma$ is called a global section if $U=M$.  Denote by $\rH^0(E)$ the linear space of global sections on $E$.
 A subspace $\V\subset \rH^0(E)$ is said to generate $E$ if
 $\V(\zeta)$, the values of all section in $\V$ at each $\zeta\in M$
 is equal to $E_\zeta$.

 The following proposition is a generalization
 of the classical Bertini's theorem in algebraic geometry,
 and it is a standard consequence of  Generic Smoothness Theorem.
 For the convenience of the reader we state and give a short proof of this proposition.
 \begin{theo}[{``Bertini-type''} theorem]\label{berttheo}
 Let $E$ be a vector bundle on $M$.
 Let $\V\subset \rH^0(E)$ be a subspace which generates $E$.
 Then
\begin{enumerate}
\item{}if $\rank E>\dim M$ for the generic $\sigma\in \V$ the zero locus of $\sigma$
 is empty.
\item{}if $\rank E\le\dim M$ for the generic $\sigma\in \V$ the zero locus of $\sigma$
 is either smooth of codimension $\rank E$ or it is empty.
 \item{}if  $\rank E=\dim M$  the zero locus of
 the generic $\sigma\in V$ consists of $\nu(E)$ simple points,
 where $\nu(E)$ is the top Chern number of $E$.
\end{enumerate}
 \end{theo}
 \proof We identify the vector bundle $E$ with its locally free sheaf of sections,
 see \cite[B.3]{Ful}.  We have the projection {$E\rig{\pi} M$},  where the fiber 
$\pi^{-1}(\zeta)$ is isomorphic to the vector
 space $E_\zeta$. Let $\Pi\subset E$ be the zero section.
 By the assumption we have a natural projection of maximal rank
 $$M\times \V\rig{p}E.$$
 Let $Z=p^{-1}(\Pi)$, then $Z$ is isomorphic to the variety
 $\{(\zeta,\sigma)\in M\times \V|\sigma(\zeta)=0\}$ and it has dimension
 equal to $\dim M+\dim \V-\rank E$.
 Consider the natural projection $Z\rig{q}\V$, now $\forall \sigma\in V$
 the fiber
 $q^{-1}(\sigma)$ is naturally isomorphic to the zero locus of $\sigma$.
 We have two cases. If $q$ is dominant  (namely the image of $q$ is dense)
 then by the Generic Smoothness theorem \cite[ Corol. III 10.7]{Har}
 $q^{-1}(\sigma)$ is smooth of dimension $\dim X-\rank E$ for generic $\sigma$.

 If $q$ is not dominant (and this always happens in the case $\rank E>\dim M$) then $q^{-1}(\sigma)$ is empty for generic $\sigma$.
 This concludes the proof of the first two parts. The third part follows from
 \cite[Example 3.2.16]{Ful}.
 \qed

For our purposes we need the following refinement of Theorem \ref{berttheo}.
\begin{defn}\label{Valmgen}
 Let $\pi: E\to M$ be a vector bundle on a smooth projective variety $M$ such that $\rank E\ge \dim M$.
 Let $\V\subset \rH^0(E)$ be a subspace.  Then $\V$  almost generates $E$ if the following conditions hold.
Either $\V$ generates $E$, (in this case $k=0$) or
there exists $k\ge 1$ smooth strict irreducible subvarieties $Y_1,\ldots,Y_k$ of $M$ satisfying the following properties. 
First, on each $Y_j$ there is a vector bundle $E_j$.
Second, after assuming $Y_0=M$ and $E_0=E$, the following conditions hold.
\begin{enumerate}
\item $\rank E_j > \dim Y_j$ for each $j\ge 1$.
\item  Let $\pi_j: E_j \to Y_j$, and for any $i,j\ge 0$
assume that $Y_i$ is a subvariety of $Y_j$.
   Then $E_i$ is a subbundle of ${E_j}_{|Y_i}$.
\item $\V(\zeta)\subset (E_j)_\zeta$ for $\zeta\in Y_j$.
\item Denote by $P_j\subset [k]$ the set of all $i\in [k]$ such that $Y_i$ are strict subvarieties of $Y_j$.
Then $\V(\zeta)=(E_j)_{\zeta}$ for $\zeta\in Y_j\setminus_{i\in P_j} Y_i$.
\end{enumerate}
\end{defn}
\begin{theo}\label{berttheom}
 Let $E$ be a vector bundle on a smooth projective variety $M$.  Assume that $\rank E\ge \dim M$.
 Let $\V\subset \rH^0(E)$ be a subspace which almost  generates $E$.
 Then
\begin{enumerate}
\item{} If $\rank E>\dim M$ then for a generic $\sigma\in \V$ the zero locus of $\sigma$
 is empty.
\item{} If  $\rank E=\dim M$ then the zero locus of
 a generic $\sigma\in \V$ consists of $\nu(E)$ simple points lying outside $\cup_{j\in [k]} Y_j$,
 where $\nu(E)$ is the top Chern number of $E$.
\end{enumerate}
 \end{theo}
\proof  Like in the proof of Theorem \ref{berttheo} we consider the variety
 $$Z=\{(\zeta,\sigma)\in M\times \V|\sigma(\zeta)=0\}.$$ 

We consider the two projections
$$\begin{array}{ccccc}&&Z\\
&\swar{p}&&\sear{q}\\
M&&&&{\bf V}\\
\end{array}$$
The fiber $q^{-1}(v)$ can be identified with the zero locus of $v$.
If $\zeta\in Y_k$, by\emph{4} of Definition \ref{Valmgen},
the fibers $p^{-1}(\zeta)$ can be identified with a subspace
of $\V$ having codimension $\rank E_k$.  It follows that the dimension of
$p^{-1}(Y_k)$ is equal to $\dim \V-\rank E_k+\dim Y_k$
which, by \emph{1} of Definition \ref{Valmgen}, is strictly smaller than $\dim \V$ if $k\ge 1$.
Let $Y=\cup_{k\ge 1}Y_k$. Then $p^{-1}(X\setminus Y)\subset Z$
is a fibration and it is smooth. Call $\overline{q}$ the restriction of $q$
to $p^{-1}(X\setminus Y)$.
If $\rank E>\dim M$ we get that $\overline{q}$
is not dominant and the generic  fiber ${\overline q}^{-1}(v)$ is empty.
If $\rank E=\dim M$, by Generic Smoothness Theorem applied to
$\overline{q}\colon p^{-1}(X\setminus Y)\to \V$,
we get that there exists $V_0\subset \V$, with $V_0$ open, such that the fiber ${\overline q}^{-1}(v)$ is smooth for $v\in V_0$.

Moreover, {the dimension count yields that} $q(p^{-1}(Y))$ is a closed proper subset of $\V$ (note that $q$ is a proper map).
Call $V_1=\V\setminus q(p^{-1}(Y))$, again open.

It follows that for $v\in V_0\cap V_1$ the fiber $q^{-1}(v)$ coincides with the fiber ${\overline{q}}^{-1}(v)$,
which is smooth by the previous argument, given by finitely many simple points.
The number of points is $\nu(E)$, again by \cite[Example 3.2.16]{Ful}.
\qed



 \section {The number of singular vector tuples of a generic tensor}\label{sec:numsingtpl}
 In this section we compute the number of singular vector tuples of a generic tensor $\cT\in \C^{\m}$.
In what follows we need the following two lemmas.  The first one is well known and we leave its proof to the reader.
Denote by $Q_m:=\{\x\in\C^m, \x\trans\x=0\}$ the quadric of isotropic vectors.
\begin{lemma}\label{lemquotspace}  Let $\x\in\C^m\setminus\{\0\}$ and denote $\U:=\C^m/[\x]$.  For $\y\in\C^m$ denote by $[[\y]]$
the element in $\U$ induced by $\y$.  Then
\begin{enumerate}
\item Any linear functional $\g:\U\to \C$ is uniquely represented by $\w\in\C^m$ such that $\w\trans\x=0$ and $\g([[\y]])=\w\trans \y$.
In particular, if $\x\in Q_m$  then the functional $\g_\x:\U\to \C$ given by $g([[\y]])=\x\trans\y$ is a linear functional.
\item Suppose that $\x\not\in Q_m$ and $a\in \C$ is given.  Then for each $\y\in\C^m$ there exists a unique 
$\z\in\C^m$ such that $[[\z]]=[[\y]]$ and $\x\trans\z=a$.

\end{enumerate}
\end{lemma}
\begin{lemma}\label{condexistT}
Let $\m=(m_1,\ldots,m_d)\in\N^d$.  Assume that $\x_i\in\F^{m_i}\setminus\{\0\}, \y_i\in\F^{m_i}$ are given for $ i\in[d]$.  
\begin{enumerate}
\item
There exists
$\cT\in\F^{\m}$ satisfying
\begin{equation}\label{Txyeq}
\cT\times \otimes_{j\in[d]\setminus\{i\}}\x_j=\y_i, 
\end{equation}
 for any $i\in [d]$ if and only if the following compatibility conditions hold.
\begin{equation}\label{Txycompcond}
\x_1\trans\y_1=\ldots=\x_d\trans\y_d.
\end{equation}
\item
Let $P\subset [d]$ be the set of all $p\in[d]$ such that $\x_p$ is isotropic.   Consider the following system of equation
\begin{equation}\label{Txyequot}
[[\cT\times \otimes_{j\in[d]\setminus\{l\}}\x_j]]=[[\y_l]], 
\end{equation}
for any $l\in [d]$.
Then there exists $\cT\in\F^{\m}$ satisfying \eqref{Txyequot} if and only if one of the following conditions hold.
\begin{enumerate}
\item $|P|\le 1$, i.e there exists at most one isotropic vector in $\{\x_1,\ldots,\x_d\}$.
\item $|P|=k\ge 2$.   Assume that $P=\{i_1,\ldots,i_k\}$.  Then
\begin{equation}\label{condisotr}
\x_{i_1}\trans \y_{i_1}=\x_{i_2}\trans \y_{i_2}=\ldots =\x_{i_k}\trans \y_{i_k}. 
\end{equation}
\end{enumerate}

\item Fix $i\in [d]$.  Let $P\subset [d]\setminus\{i\}$ be the set of all $p\in[d]\setminus\{i\}$ such that $\x_p$ is isotropic.
Then there exists $\cT\in\F^{\m}$ satisfying the condition \eqref{Txyeq} and the conditions \eqref{Txyequot} for all $l\in[d]\setminus\{i\}$
if and only if one of the following conditions hold.
\begin{enumerate}
\item $|P|=0$.
\item $|P|=k-1 \ge 1$.  Assume that $P=\{i_1,\ldots,i_{k-1}\}$.  Let $i_k=i$.  Then \eqref{condisotr} hold.
\end{enumerate}
\end{enumerate}
\end{lemma}
\begin{proof}  
\emph{1}.  Assume first that \eqref{Txyeq} holds.  Then $\cT\times \otimes_{j\in[d]}\x_j=\x_i\trans\y_i$ for $i\in[d]$.  Hence \eqref{Txycompcond} holds.
Suppose now that \eqref{Txycompcond} holds.  We now show that there exists $\cT\in\F^\m$ satisfying  \eqref{Txyeq}.

Let $U_j=[u_{pq,j}]_{p=q=1}^{m_j}\in\gl(m_j,\F)$ for $j\in [d]$.  Let $U:=\otimes_{i\in[d]}U_i$.
Then $U$ acts on $\F^{\m}$ as a matrix acting on the corresponding vector space.   That is, let  $\cT'=U\cT$ and assume that 
$\cT=[t_{i_1,\ldots,i_d}], \cT'=[t'_{j_1,\ldots,j_p}]$.  Then
\[t_{j_1,\ldots,j_d}'=\sum_{i_1\in[m_1], \ldots,i_d\in[m_d]} u_{j_1i_1,1}\ldots u_{j_di_d,d} t_{i_1,\ldots,i_d}, \quad j_1\in[m_1],\ldots,j_d\in[m_d].\]
The conditions \eqref{Txyeq} for $\cT'$ become
\begin{equation}\label{Txyeq'}
\cT'\times \otimes_{j\in[d]\setminus\{i\}}\x_j'=\y_i', \; i\in[d],\quad \x'_i=(U_i\trans)^{-1}\x_i,\;\y_i'=U_i\y_i,\; i\in [d].
\end{equation}
Clearly, $\x_i\trans\y_i=(\x_i')\trans \y_i'$ for $i\in[d]$.  Since $\x_i\ne 0$ there exists $U_i\in\gl(m_i,\F)$ such that 
$(U_i\trans)^{-1}\x_i=\e_{1,i}=(1,0,\ldots,0)\trans$
for $i\in [d]$.  Hence it is enough to show that \eqref{Txyeq} is satisfied for some $\cT$ if $\x_i=\e_{i,1}$ for $i\in[d]$ if
$\e_{1,1}\trans\y_1=\ldots=\e_{d,1}\trans\y_d$.  Let $\y_i=(y_{1,i},\ldots,y_{m_i,i})\trans$ for $i\in[d]$.  Then the conditions 
\eqref{Txycompcond} imply that
$y_{1,1}=\ldots=y_{1,d}$.  
Choose a suitable $T=[t_{i_1,\ldots, i_d}]$ as follows.
$t_{i_1,\ldots, i_d}=y_{i_j,j}$ if $i_k=0$ for $k\neq j$, $t_{i_1,\ldots, i_d}=0$ otherwise.
Then \eqref{Txyeq} holds.

\emph{2}.  We now consider the system \eqref{Txyequot}.  This system is solvable if and only we can find $t_1,\ldots,t_d\in\F$ such that
\begin{equation}\label{solvconquot}
\x_1\trans(\y_1+t_1\x_1)=\ldots=\x_d\trans(\y_d+t_d\x_d).
\end{equation}

Suppose first that $\x_i\not\in Q_{m_i}$ for  $i\in [d]$.  Fix $a\in\F$.  Choose $t_i=\frac{a-\x_i\trans \y_i}{\x_i\trans\x_i}$ for $i\in[d]$.
Hence  the system \eqref{Txyequot} is solvable.   Suppose next that $\x_j\in Q_{m_j}$.  Then $\x_j\trans(\y_j+t_j\x_j)=\x_j\trans\y_j$.
Assume that $P=\{j\}$.  Let $a=\x_j\trans\y_j$.   Choose $t_i, i\ne j$ as above to deduce that \eqref{solvconquot} holds.
Hence \eqref{Txyequot} is solvable.

Assume finally that $k\ge 2$ and $P=\{i_1,\ldots,i_k\}$. \eqref{solvconquot} yields that if  \eqref{Txyequot} is solvable then \eqref{condisotr} holds.
Suppose that  \eqref{condisotr} holds.  Let $a=\x_{i_1}\trans \y_{i_1}=\ldots =\x_{i_k}\trans \y_{i_k}$.  
For $i\not\in P$ let $t_i=\frac{a-\x_i\trans \y_i}{\x_i\trans\x_i}$ to deduce that 
the condition \eqref{solvconquot} holds.  Hence \eqref{Txyequot} is solvable.

\emph{3}.  Consider the equation \eqref{Txyeq} and the equations \eqref{Txyequot}  for $l\in [d]\setminus\{i\}$.  Then this system is solvable iff and only if 
the system \eqref{solvconquot} is solvable for $t_i=0$ and some $t_l\in\F$ for $l\in[d]\setminus\{i\}$.  Let $a=\x_i\trans \y_i$.
Assume that $|P|=0$.  Choose $t_l=\frac{a-\x_l\trans \y_l}{\x_l\trans\x_l}$ for $l\in[d]\setminus\{i\}$ as above to deduce that this system is solvable.
Assume that $P=\{i_1,\ldots,i_{k-1}\}$ for $k\ge 2$.  Suppose this system is solvable for some $\cT\in\F^\m$.  Then $a=\x_j\trans\y_j$
for each $j\in P$.  Let $i_k:=i$.   Hence \eqref{solvconquot} holds.  Vice versa assume that \eqref{solvconquot} holds.  Choose
$t_l=\frac{a-\x_l\trans \y_l}{\x_l\trans\x_l}$ for  $l\not\in P\cup\{i\}$.  Then  \eqref{solvconquot} holds.  Hence our system is solvable.
\qed
\end{proof}

 \begin{lemma}\label{globsecR}
 Let $R(i,\m)$ and $R(\m)$ be the vector bundles over the Segre product $\Pi(\m)$ defined in \eqref{defRm}.
 Denote by $\rH^0(R(i,\m))$ and $\rH^0(R(\m))$ the linear space of global sections of $R(i,\m)$  and $R(\m)$ respectively.
 Then the following conditions hold.
 \begin{enumerate}
 \item\label{globsecR1}
 For each $i\in [d]$ there  exists  a monomorphism $L_i:\C^{\m}\to \rH^0(R(i,\m))$ such that
 $L_i(\C^{\m})$ generates $R(i,\m)$ (see \S\ref{subsec:bertini}).
 \item\label{globsecR2}
 $L=(L_1,\ldots,L_d)$ is a monomorphism of the direct sum of $d$ copies of $\C^{\m}$,  ({denoted as} $\oplus^d \C^{\m}$,) to
 $\rH^0(R(\m))$ {which generates $R(\m)$.}
\item\label{globsecR3}  Let $\delta: \C^{\m}\to \oplus^d \C^{\m}$ be the diagonal map $\delta(\cT)=(\cT,\ldots,\cT)$.
Consider $([\x_1],\ldots, [\x_d])\in \Pi(\m)$.  
\begin{enumerate}\item \label{globsecR3a}
If at most one of $\x_1,\ldots,\x_d$ is isotropic then  
$L\circ\delta(\C^{\m})$ (as a space of sections of $R(\m)$) generates $R(\m)$ at $([\x_1],\ldots,[\x_d])$.
\item\label{globsecR4}  Let $P\subset [d]$ be the set of all $i\in[d]$
such that $\x_i$ is isotropic.  Assume that $P=\{i_1,\ldots,i_k\}$ where $k\ge 2$.
 Let $\g_{\x_{i_p}}$ be the linear functional on the fiber of $\pi_{i_p}^*Q(m_{i_p})$ at $([\x_1],\ldots,[\x_d])$ as defined 
in Lemma \ref{lemquotspace} for $p=1,\ldots,k$.  
Let $\U(P)$ be the subspace of all  linear transformations 
$\tau=(\tau_1,\ldots,\tau_d)\in R(\m)_{([\x_1],\ldots,[\x_d])}, \tau_i\in R(i,\m)_{([\x_1],\ldots,[\x_d])}, i\in [d]$ satisfying
\begin{equation}\label{defU(P)}
\g_{\x_{i_1}}(\tau_{i_1}(\otimes_{j\in[d]\setminus\{i_1\}} \x_j))=\ldots=\g_{\x_{i_k}}(\tau_{i_k}(\otimes_{j\in[d]\setminus\{i_k\}} \x_j)).
\end{equation}
Then $L\circ\delta(\cT)([\x_1],\ldots,[\x_d])\in \U(P)$ for each $\cT\in \C^\m$
Furthermore,  $L\circ\delta(\C^{\m})([\x_1],\ldots,[\x_d])=\U(P)$.
\end{enumerate}
 \end{enumerate}
 \end{lemma}
 \begin{proof} 
For $\cT\in\C^{\m}$  we define the section $ L_i(\cT)(([\x_1],\ldots,[\x_d]))\in R(i,\m)_{([\x_1],\ldots,[\x_d])}$ as follows:
 \begin{equation}\label{defLi}
 L_i(\cT)(([\x_1],\ldots,[\x_d]))(\otimes_{j\in[d]\setminus\{i\}}\x_j):=[[\cT\times \otimes_{j\in[d]\setminus\{i\}}\x_j]].
 \end{equation}
 It is straightforward to check that $L_i(\cT)$ is a global section of $R(i,\m)$. 

 Assume $\cT\ne 0$.  Then there exist $\vv_j\in\C^{m_j}, j\in [d]$ such that $\cT\times \otimes_{j\in[d]}\vv_j\ne 0$.
 Hence $\uu_i:=\cT\times \otimes_{j\in[d]\setminus\{i\}}\vv_j\in \C^{m_i}\setminus\{\0\}$.
 Let $\x_j=\vv_j$ for $j\ne i$.  Choose $\x_i\in \C^{m_i}\setminus\{[\uu_i]\}$.  Then 
$L_i(\cT)(([\x_1],\ldots,[\x_d]))\ne 0$.
 Hence $L_i$ is injective.

 {We now show that  $L_i(\C^{\m})$ generates $R(i,\m)$.  Let $\y_i\in\C^{m_i}$.  Choose $\g_j\in\C^{m_j}$ such that
$\g_j\trans\x_j=1$ for $j\in[d]$.  Set $\cT=(\otimes_{j\in[i-1]}\g_j)\otimes \y_i\otimes (\otimes_{j\in [d]\setminus [i]} \g_j)$.
Then $ L_i(\cT)(([\x_1],\ldots,[\x_d]))= [[\y_i]]$.  This shows \emph{1}.
}

 Define $L((\cT_1,\ldots,\cT_d))(([\x_1],\ldots,[\x_d]))=\oplus_{i\in[d]}L_i(\cT_i)(([\x_1],\ldots,[\x_d]))$.
 Then 

\noindent
$L((\cT_1,\ldots,\cT_d))\in \rH^0(R(\m))$.  Clearly $L$ is a monomorphism.  {Furthermore $L(\oplus^d \C^\m)$
generates $\R(\m)$.  This shows \emph{2}.}

{The cases \emph{\ref{globsecR3a}} and  \emph{\ref{globsecR4}} of our Lemma follow from parts \emph{2a} and  
 \emph{2b}  of Lemma \ref{condexistT} respectively.}
  \qed
\end{proof}

K\"unneth formula \cite{Hir} yields the equalities
 \begin{equation}\label{eql=H}
 L_i(\C^\m)=\rH^0(R(i,\m)), \;i\in[d], \quad L(\oplus^d \C^{\m})=H^0(R(\m)).
 \end{equation}

\begin{corol}\label{allisotropic}
Assume that $([\x_1],\ldots,[\x_d])\in \Pi(\m)$ is a singular $d$-tuple of a tensor $\cT$
corresponding to a nonzero singular value.   Then one of the following holds.
\begin{enumerate}
\item\label{lem3.2a} All $\x_i$ are isotropic.
\item\label{lem3.2b} All $\x_i$ are non isotropic.
\end{enumerate}
\end{corol}
 For $\cT\in\R^{\m}$ with a real singular vector tuple
 $([\x_1],\ldots,[\x_d])\in \Pi_\R(\m)$ the condition $\prod_{i\in[d]}\lambda_i=0$ implies that $\lambda_i=0$ for each $i$.
 { Indeed, since $\x_i\in \R^{m_i}\setminus \{\0\}$ it follows from \eqref{defsingtpl} that 
$\lambda_i=\frac{\cT\times \otimes_{j\in [d]} \x_j}{\x_i\trans\x_i}$ for each $i\in [d]$.  
Thus $\lambda_k=0$ for some $k\in [d]$ yields that $\cT\times \otimes_{j\in [d]} \x_j=0$.  Hence
 each $\lambda_i=0$.}

 However, this observation is not valid for complex tensors,
already in the case of complex-valued matrices ($d=2$), {see an example below}.
 It is straightforward to see that a singular value pair $([\x_1],[\x_2])$ of $A\in\C^{m_1\times m_2}$ is given by the following conditions
 \begin{equation}\label{defsingpair}
 A\x_2=\lambda_1\x_1, \quad A\trans \x_1=\lambda_2\x_2, \quad \x_i\in\C^{m_i}\setminus\{\0\}, \lambda_i\in\C, \quad i=1,2.
 \end{equation}
 Consider the following simple example.
 \[A=\left[\begin{array}{cc} 1&\ii\\-\ii&1\end{array}\right],\;
 \x_1=\left[\begin{array}{c} 1\\0\end{array}\right],\;
 \x_2=\left[\begin{array}{c} 1\\ \ii\end{array}\right].\]
 Then $A\trans \x_1=\x_2, A\x_2=\0$, i.e. $\lambda_1=1,\lambda_2=0$.

\begin{lemma}\label{singularsection}
 Let $\cT\in\C^\m$ and consider the section $\hat\cT:=L\circ \delta(\cT) \in\rH^0(R(\m))$.   We have that 
$([\x_1],\ldots,[\x_d])\in \Pi(\m)$ is a zero of $\hat\cT$ if and only if $([\x_1],\ldots,[\x_d])$ is a singular vector tuple corresponding to $\cT$.
\end{lemma}
\proof
Suppose first that $\hat \cT(([\x_1],\ldots,[\x_d]))=0$.  Then $L_i(\cT)(([\x_1],\ldots,[\x_d]))$ is a zero vector in the fiber $R(i,\m)$
at $([\x_1],\ldots,[\x_d])$.  
Suppose first that $\cT\times (\otimes_{j\in[d]\setminus\{i\}}\x_j)\ne \0$.  Then
 $\cT\times (\otimes_{j\in[d]\setminus\{i\}}\x_j)=\lambda_i\x_i$ for some $\lambda_i\ne 0$.  Otherwise the previous equality holds with $\lambda_i=0$.
 Hence $([\x_1],\ldots,[\x_d])$ is a singular vector tuple corresponding to $\cT$.
 Vice versa, it is straightforward to see that if $([\x_1],\ldots,[\x_d])$ is a singular vector tuple corresponding to $\cT$ then the section
 $\hat\cT$ vanishes at $([\x_1],\ldots,[\x_d])\in \Pi(\m)$.\qed

 We now bring the proof of Theorem \ref{numbsingtup},  which was stated in \S1.\\

 {\bf{Proof of Theorem \ref{numbsingtup}.}}  
Let $V=L\circ\delta(\C^{\m})$ be the subspace of sections of $R(\m)$ given by tensors (embedded diagonally).
We now show that $\V$ almost generates $R(\m)$ as defined in Definition \ref{Valmgen}.  
First, $\rank R(\m)=\dim \Pi(\m)$.  Second, let $2^{[d]_k}$ be the set of all subsets of $[d]$ of cardinality $k$ for each $k\in [d]$.
Let $\alpha\in 2^{[d]_k}$.  Define $Y_\alpha=X_1\times\ldots X_d$, where $X_i=\P(Q_{m_i})$ if $i\in\alpha$ and $X_i=\P(\C^{m_i})$
otherwise.  Clearly, $Y_\alpha$ is a strict smooth subvariety of $\Pi(\m)$ of codimension $k$.  
Note that $Y_\beta \subsetneq Y_\alpha$ if and only if $\alpha \subsetneq \beta$.
We now define the subbundle $E_\alpha$ of $\pi^{-1}(Y_{\alpha})$.  If $\alpha\in 2^{[d]_1}$ then $E_\alpha=\pi^{-1}(Y_{\alpha})$.  
Assume now that $k>1$.  Let $\alpha=\{i_1,\ldots,i_k\}$.
Let $([\x_1],\ldots,[\x_d])\in Y_{\alpha}$.  So $\x_{i_l}\in Q_{m_{i_l}}$ for $l=1,\ldots,k$.
Then the fiber $E_{\alpha}$ at $([\x_1],\ldots,[\x_d])$ is the set of all vectors satisfying
\eqref{defU(P)}.
Note that $\rank E_{\alpha}=\dim Y_\alpha+1$.   
Assume that $\alpha \subsetneq \beta$.  Clearly, $E_\beta$ is a strict subbundle of $\pi_\alpha^{-1}(Y_\beta)$.
Hence the conditions \emph{1-2} of Definition \ref{Valmgen} hold.  
Lemma \ref{globsecR} implies the conditions \emph{3-4} of Definition \ref{Valmgen} hold. 
Theorem \ref{berttheom} implies that for a generic $\cT\in\C^\m$ the section $L\circ \delta(\cT)$ has a 
finite number of simple zeros.   Moreover, this number is equal to the top Chern number of $R(\m)$.  
{Lemma \ref{cmfor}}  yields that the top Chern number of $R(\m)$ is $c(\m)$.

 It is left to show that a generic $\cT\in \C^\m$ does not have a zero singular value.
 Fix $i\in[d]$ and consider the set of all $\cT\in\C^\m$ which have a singular vector tuple $([\x_1],\ldots,[\x_d])\in \Pi(\m)$
 with $\lambda_i=0$.  

Let $R(i,\m)'$ and $R_i(\m)'$ be defined in \eqref{defRm}. 
 Similar to the definition \eqref{defLi}, we can define a monomorphism $L_i':\C^\m\to \rH^0(R(i,m)')$ by the equality
 \[L_i'(\cT)(([\x_1],\ldots,[\x_d]))(\otimes_{j\in[d]\setminus\{i\}}\x_j):=\cT\times \otimes_{j\in[d]\setminus\{i\}}\x_j.\]
 Let $\tilde L_i=(L_1,\ldots,L_{i-1},L_i',L_{i+1},\ldots,L_d):\oplus_{j\in[d]}\C^\m\to \rH^0(R_i(\m)')$.

 {We claim that $\tilde L_i\circ \delta(\C^\m)$ almost generates $R_i(\m)'$.   Clearly, $\rank R_i(\m)'=\dim \Pi(\m)+1$.    
Recall that a vector in $(\tau_1,\ldots,\tau_d)\in R_i(\m)'_{([\x_1],\ldots,[\x_d])}$ is of the form 
\begin{equation}\label{secRim'}
 \tau_j: \hat T(\m_j)\to \pi_j^*Q(m_j) \textrm{ for }j\in [d]\setminus\{i\}, \quad \tau_i: \hat T(\m_i)\to \pi_j^*F(m_i).
\end{equation}

Let $\alpha\subset [d]\setminus\{i\}$ be a nonempty set.
Then $Y_\alpha=X_1\times \ldots\times X_d$, where $X_j=\P(Q_{m_j})$ if $j\in \alpha$ and $X_j=\P(\C^{m_j})$ if $j\not\in \alpha$.  
(Note that $X_i=\P(\C^{m_i})$.)  
We now define the vector bundles $\pi_{\alpha}:E_\alpha\to Y_{\alpha}$.  Let $\pi: R_i(\m)'\to \Pi(\m)$.
Assume that $\alpha=\{i_1,\ldots,i_{k-1}\}\subset [d]\setminus \{i\}$ where $k-1\ge 1$.  Then $E_\alpha$ is the subbundle 
$\pi^{-1}(Y_{\alpha})$ defined as follows.  For $([\x_1],\ldots,[\x_d])\in Y_{\alpha}$ it consists of all sections of the form \eqref{secRim'}
satisfying a variation of the condition \eqref{defU(P)}:
\[\g_{\x_{i_1}}(\tau_{i_1}(\otimes_{j\in[d]\setminus\{i_1\}} \x_j))=\ldots=\g_{\x_{i_{k-1}}}(\tau_{i_{k-1}}
(\otimes_{j\in[d]\setminus\{i_{k-1}\}} \x_j)) =\x_{i}\trans \tau_{i}(\otimes_{j\in[d]\setminus\{i\}} \x_j).\]

Note that $\rank E_\alpha=\dim Y_\alpha+1$.  Clearly, the conditions 
of \emph{1-2} of Definition \ref{Valmgen} hold.  
Part \emph{3} of Lemma \ref{condexistT} implies the conditions \emph{3-4} of Definition \ref{Valmgen}. 
Theorem \ref{berttheom} yields that a generic section of $\tilde L_i\circ \delta(\cT)$ does not have zero.
 Thus, $\cT$ does not have a singular vector tuple satisfying
 \eqref{defsingtpl} with $\lambda_i=0$.  Hence a generic tensor $\cT\in\C^\m$ does not have a zero singular value.}

 Clearly, a generic $\cT\in\R^\m$ has exactly $c(\m)$ simple {complex-valued} singular value tuples.  Only some of those can be realized
 as points in $\Pi_\R(\m)$.
 \qed

 We first observe that Theorem \ref{numbsingtup} agrees with the standard theory of singular values for $m\times n$ real matrices.
{Namely, a generic $A\in\R^{m\times n}$ has exactly $\min(m,n)$ nonzero singular values which are all positive and pairwise
 distinct.  The corresponding singular vector pairs are simple.

 We now point out a matrix proof of Theorem \ref{numbsingtup} for $d=2$.
 Let $\rO(m)\subset \C^{m\times m}$ be the variety of $m\times m$  orthogonal matrices and $\rD_{m,n}\subset \C^{m\times n}$
 the linear subspace of all diagonal matrices.
 Consider the trilinear polynomial map {$F:\rO(m_1)\times \rD_{m_1,m_2}\times \rO(m_2)\to\C^{m_1\times m_2}$} given by
 $(U_1,D,U_2)\mapsto U_1 D U_2\trans$.   Singular value decomposition yields that any $A\in \R^{m_1\times m_2}$ is of the form
 $U_1 D U_2\trans$, where $U_1,U_2$ are real orthogonal and $D$ is a nonnegative diagonal matrix.  Hence 
 $F(\rO(m_1)\times \rD_{m_1,m_2}\times \rO(m_2))=\R^{m_1\times m_2}$.  Therefore the image of $F$ is dense in $\C^{m_1\times m_2}$.  
 Hence a generic $A\in\C^{m_1\times m_2}$ is of the form $U_1\trans D U_2$.   Furthermore, we can assume that $D=\diag(\lambda_1,\ldots,\lambda_l),  
 l=\min(m_1,m_2)$, where the diagonal entries are nonzero and pairwise distinct.  Assume that $\x_i,\y_i$ are the $i-th$ columns of    
 of $U_1,U_2$ respectively for $i=1,\ldots,l$.  Then $([\x_i],[\y_i])$ is a simple singular value tuple corresponding to $\lambda_i$ for $i=1,\ldots,l$.}

 We list for the convenience of the reader a few values $c(\m)$.
 First
 \begin{equation}\label{c2d}
 c(\underbrace {2,\ldots , 2}_d)=d!
 \end{equation} 
 {Indeed, $\frac{\hat t_i^2-t_i^2}{\hat t_i-t_i}=(\hat t_i+t_i)=\sum_{j\in[d]} t_j$.  Therefore
 $\prod_{j\in[d]} \frac{\hat t_i^2-t_i^2}{\hat t_i-t_i}= (\sum_{j\in[d]} t_j)^d$. 
 Clearly, the coefficient of $t_1\ldots t_d$ in this polynomial is $d!$.}

 Second, we list at next page the first values in the case $d=3$.
 From this table one sees that $c(m_1,m_2,m_3)$
 stabilizes for $m_3\ge m_1+m_2-1$, the case when equality holds
 is called the boundary format case in the theory of hyperdeterminants
 (\cite{GKZ}). It is the case where a ``diagonal'' naturally occurs,
 like in the following figure:

 \begin{figure}[h]
      \centering
      \includegraphics[width=50mm]{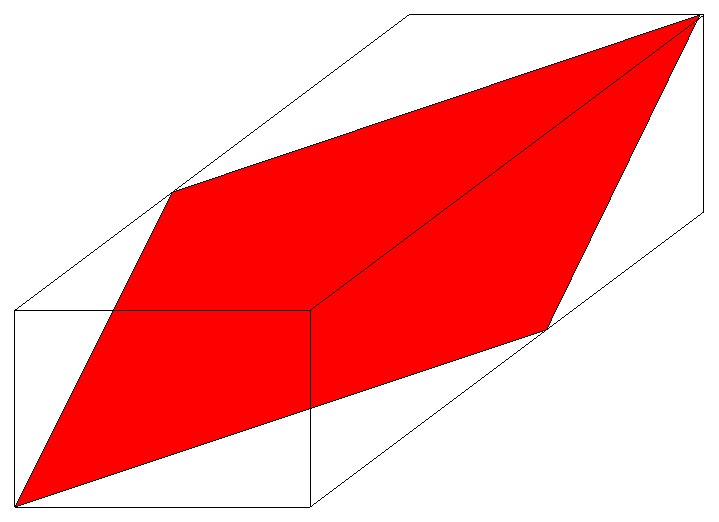}
      \caption{A diagonal in the three dimensional case}
      \label{fig1}
 \end{figure}


 In $d=2$ case, boundary format means square.
 
 \begin{table}
 \centering
 $\begin{array}{r|r|c}d_1, d_2, d_3&c(d_1,d_2,d_3)\\
\hline
 2, 2, 2& 6\\
 2, 2, n& 8&n\ge 3\\
 2, 3, 3& 15\\
 2, 3, {n}& 18&n\ge 4\\
 2, 4, 4& 28\\
 2, 4, n& 32&n\ge 5\\
 2, 5, 5& 45\\
 2, 5, n&50&n\ge 6\\
 2, m, m+1&2m^2\\
 3, 3, 3& 37\\
 3, 3, 4& 55\\
 3, 3, n& 61&n\ge 5\\
 3, 4, 4& 104\\
 3, 4, 5& 138\\
 3, 4, n& 148&n\ge 6\\
 3, 5, 5& 225\\
 3,5,6&280\\
 3,5,n&295&n\ge 7\\
 3, m, m+2&\frac{8}{3}m^3-2m^2+\frac{7}{3}m\\
 4, 4, 4& 240\\
 4, 4, 5& 380\\
 4,4,6&460\\
 4,4,n&480&n\ge 7\\
 4, 5, 5& 725\\
 4,5,6&1030\\
 4,5,7&1185\\
 4,5,n&1220&n\ge 8\\
 5, 5, 5& 1621\\
 5,5,6&2671\\
 5,5,7&3461\\
 5,5,8&3811\\
 5,5,n&3881&n\ge 9\\

 \end{array}$
 \caption{Values of $c(d_1,d_2,d_3)$}
 \label{fig1}
 \end{table}
  
 \section{Partially symmetric singular value tuples}\label{sec:parsymsingvt}
 For an integer $m\ge 2$ let $m^{\times d}:=(\underbrace{m,\ldots,m}_d)$.  Then $\cT\in\F^{m^{\times d}}$ is called $d$-cube, or simply a cube tensor.  
Denote by $\rS^d(\F^{m})\subset\F^{m^{\times d}}$ the subspace of symmetric tensors.  For $\cT\in\rS^d(\F^{m})$ it is natural
 to consider a singular value tuple \eqref{defsingtpl} where $\x_1=\ldots=\x_d=\x$ \cite[formula (7) with $p=2$]{Lim05}.
 This is equivalent to the system
 \begin{equation}\label{nleigprb}
 \cT\times \otimes^{d-1}\x=\lambda\x,\quad \x\ne 0.
 \end{equation}
 Here $\otimes^{d-1}\x:=\underbrace{\x\otimes\ldots\otimes\x}_{d-1}$.
 Furthermore, the contraction in \eqref{nleigprb} is on the last $d-1$ indices.  The equation \eqref{nleigprb} makes sense for any cube
 tensor $\cT\in\C^{m^{\times d}}$ \cite{Lim05,NQWW,Qi07}.  For $d=2$ $\x$ is an eigenvector of the square matrix $\cT$
 Hence for a $d$-cube tensor ($d\ge 3$) $\x$ is referred as a nonlinear eigenvalue of $\cT$.  Abusing slightly our notation we call $([\x],\ldots,[\x])
 \in \Pi(m^{\times d})$ a symmetric singular value tuple of $\cT$.  (Note that if $\cT\in\rS^d(\C^m)$ then $([\x],\ldots,[\x])$ 
is a proper symmetric singular value tuple of $\cT$.)

 Let $s_{d-1}(\cT)=[t'_{i_1,\ldots,i_d}]$ be the symmetrization of a $d$-cube $\cT=[t_{i_1,\ldots,i_d}]$ with respect to the last $d-1$ indices
 \begin{equation}\label{sd-1sym}
 t_{i_1,\ldots,i_d}'=\frac{1}{p(i_2,\ldots,i_d)} \sum_{\{j_2,\ldots,j_d\}=\{i_2,\ldots,i_d\}} t_{i_1,j_2,\ldots,j_d}.
 \end{equation}
 Here $p(i_2,\ldots,i_d)$ is the number of multisets $\{j_2,\ldots,j_d\}$ which are equal to $\{i_2,\ldots,i_d\}$.
 (Note that for $d=2$ $s_1(\cT)=\cT$.)
 It is straightforward to see that
 \begin{equation}\label{eqactcT}
 \cT\times \otimes^{d-1}\y= s_{d-1}(\cT)\otimes^{d-1}\y \textrm{ for all } \y.
 \end{equation}
 Hence in \eqref{nleigprb} we can assume that $\cT$ is symmetric with respect to the last $d-1$ indices.

 As for singular value tuples we view the eigenvectors of $\cT$ as elements of $\P(\C^m)$.
 It was shown by Cartwright and Sturmfels \cite{CS} that a generic $\cT\in \C^{m^{\times d}}$
 has exactly $\frac{(d-1)^{m}-1}{d-2}$ distinct eigenvectors.
 (This formula was conjectured in \cite{NQWW}.)

 The aim of this section is to consider ``partially symmetric singular vectors" and their numbers for a generic tensor.
 This number will interpolate our formula $c(\m)$ for the number of singular value tuples for a generic  $\cT\in\C^\m$ and the number of
 eigenvalues of generic $\cT\in\C^{m^{\times d}}$ given in \cite{CS}.

 Let $d=\omega_1+\ldots +\omega_p$
 be a partition of $d$.  So each $\omega_i$ is a positive integer.  Let $\omega_0=m_0'=0$,
 $ \bomega=(\omega_1,\ldots,\omega_p)$  and denote by $\m(\bomega)$ the $d$-tuple
 \begin{equation}\label{defmomega}
 \m(\bomega)=(\underbrace{m_1',\ldots,m_1'}_{\omega_1},\ldots,\underbrace{m_p',\ldots,m_p'}_{\omega_p})=(m_1,\ldots,m_d).
 \end{equation}
 Denote by $\rS^{\bomega}(\F)\subset \F^{\m(\bomega)}$ the subspace of tensors which are partially symmetric with respect to the
 partition $\bomega$.  That is the entries of $\cT=[t_{i_1,\ldots,i_d}]\in \rS^{\bomega}(\F)$ are invariant, if we permute indices
 in the $k-th$ group of indices $[\sum_{j=0}^k \omega_j]\setminus [\sum_{j=0}^{k-1} \omega_j]$
 for $k\in[p]$.
 Note that $\rS^{\bomega}(\F)=\rS^d(\F^\m))$ for $p=1$ and $\rS^{\bomega}(\F)=\F^\m$ for $p=d$.
 We call $\bomega=(1,\ldots,1)$, i.e. $p=d$, the \emph{trivial} partition.

 For simplicity of notation we let $\rS^{\bomega}:=\rS^{\bomega}(\C)$.
 Assume that $\cT\in \rS^{\bomega}$.  Consider a singular vector tuple $([\x_1],\ldots,[\x_d])$ satisfying \eqref{defsingtpl}
 and $\bomega$-symmetric conditions
 \begin{equation}\label{bomegsymcon}
 \x_j=\z_k \textrm{ for } j\in [\sum_{i=1}^k \omega_i m_i']\setminus [\sum_{i=0}^{k-1} \omega_i m_i'],
 \;k\in[p].
 \end{equation}
 We rewrite \eqref{defsingtpl} for an $\bomega$-symmetric singular vector tuple $([\x_1],\ldots,[\x_d])$ as follows.
 Define
 \begin{equation}\label{defmoemg}
 \otimes_{l\in [p]} (\otimes^{\omega_l-\delta_{lk}} \z_l): =\otimes_{j\in [d]\setminus\{1+\sum_{i=0}^{k-1} \omega_i m_i'\}} \x_j,
 \textrm{ for } k\in [p].
 \end{equation}
 Hence our equations for an $\bomega$ symmetric singular value tuple for $\cT\in\rS^{\bomega}$ is given by
 \begin{equation}\label{alphsymsingtpl}
 \cT\times \otimes_{l\in [p]} (\otimes^{\omega_l-\delta_{lk}} \z_l)=\lambda_k \z_k \quad k\in [p].
 \end{equation}
 In view of the definition of $\otimes_{l\in [p]} (\otimes^{\omega_l-\delta_{lk}} \z_l)$ we agree that the contraction on the left-hand side
 of \eqref{alphsymsingtpl} is done on all indices except the index $1+\sum_{i=0}^{k-1} \omega_i m_i'$.
 As for the $d$-cube tensor the system \eqref{alphsymsingtpl} makes sense for any $\cT\in\C^{\m(\bomega)}$.

 Let $\m':=(m_1',\ldots,m_p')$.
 We call $([\z_1],\ldots,[\z_p])\in\Pi(\m')$ satisfying \eqref{alphsymsingtpl}  $\bomega$-symmetric singular value tuple of $\cT\in \C^{\m(\bomega)}$.
 We say that $([\z_1],\ldots,[\z_p])$ corresponds to a zero (nonzero) singular value if $\prod_{i=1}^ p \lambda_i=0 \; (\ne 0)$.

 The aim of this section to generalize Theorem \ref{numbsingtup} to tensors in $\rS^{\bomega}$.
  \begin{theo}\label{numbsingtuppsym}
 Let $d\ge 3$ be a an integer and assume that $\bomega=(\omega_1,\ldots,\omega_p)$ is a partition of $d$.
 Let $\m(\bomega)$ be defined by \eqref{defmoemg}.
 Denote by $\rS^{\bomega}\subset \C^{\m(\bomega)}$
 the subspace of tensors partially symmetric with respect to $\bomega$.
 Let $c(\m',\bomega)$ be the coefficient of the monomial $\prod_{i=1}^ p t_i^{m_i'-1}$ in the polynomial
\begin{equation}\label{cmalphfor}
 \prod_{i\in[p]} \frac{\hat t_i^{m_i'}-t_i^{m_i'}}{\hat t_i-t_i}, \quad \hat t_i=(\omega_i-1)t_i+\sum_{j\in[p]\setminus\{i\}} \omega_jt_j, \;i\in [p].
 \end{equation}
 A generic $\cT\in \rS^{\bomega}$ has exactly $c(\m',\bomega)$
  simple $\bomega$-symmetric singular vector tuples which correspond to nonzero singular values.
  A generic $\cT\in \rS^{\bomega}$ does not have a zero singular value.
 In particular, a generic real-valued tensor $\cT\in \rS^{\bomega}_{\R}$ has at most $c(\m',\bomega)$ real singular value tuples and all of them
 are simple.   \end{theo}
 \proof 
 The proof of this theorem is analogous to the proof of Theorem \ref{numbsingtup}, so we point out briefly the needed modifications.
 Let $H(m_i')$,  $Q(m_i')$ and $F(m_i')$ be the vector bundles defined  in \S\ref{subsec:segre}.
 Let $\pi_i$ be the projection of $\Pi(\m')$ on the component $\P(\C^{m_i'})$.
 Then $\pi_i^* H(m_i'), \pi_i^* Q(m_i'),\pi_i^*F(m_i')$ are the pullbacks of the vector bundles $H(m_i'),Q(m_i'),F(m_i')$ to $\Pi(\m')$ respectively.
 Clearly $c(\pi_i^* H(m_i'))=1+t_i$ and moreover $c(\otimes ^k \pi_i^* H(m_i'))=1+kt_i$,
where $t_i^{m_i'}=0$.

 We next observe that we can view $\Pi(\m')$ as a submanifold of $\Pi(\m(\bomega)$ by using the imbedding 
\begin{equation}
\eta:\Pi(\m')\to \Pi(\m(\bomega)), \quad
 \eta(([\z_1],\ldots,[\z_p]))= ([\x_1],\ldots,[\x_d]) ,
\end{equation}
where we assume the relations \eqref{bomegsymcon}.
 Let $\tilde R(i,\m')$ and $\tilde R(i,\m')'$ be the pullback of $R(j,\m)$ and $R(j,\m)'$ respectively , where $j=1+\sum_{k=0}^{i-1} \omega_k m_k'$.
 (See  \eqref{defRm}.)
 Then 
 \begin{eqnarray} \notag
 &&\tilde R(i,\m'):=\textrm{Hom}(\eta^*\hat T(\m_j), \pi_i^*Q(m_i')), \quad \tilde R(i,\m')':=\textrm{Hom}(\eta^*\hat T(\m_j), \pi_i^*F(m_i')),\\
 \label{defRim}\\
 && \tilde R(\m'):=\oplus_{i\in[p]} \tilde R(i,\m'), \quad  \tilde R_i(\m')':=(\oplus_{j\in[p]\setminus\{i\}} \tilde R(i,\m'))\oplus  \tilde R(i,\m')' .\notag
 \end{eqnarray}
 Note that 
 \begin{eqnarray*}
 &&\rank \tilde R(i,\m')=\rank \tilde R(i,\m')'-1=m_i'-1, \\  
 &&\rank \tilde R(\m')= \rank \tilde R_i(\m')'-1=\dim \Pi(\m').
 \end{eqnarray*}
 As in the proof of Lemma \ref{cmfor} we deduce that the top Chern class of $\tilde R(i,\m')$ is given by the polynomial
 \begin{equation}\label{topchernpotRim}
 \sum_{j=0}^{m_i'-1} (\sum_{k\in [p]} (\omega_k-\delta_{ki})t_k)^j t_i^{m_i'-1-j}), \quad i\in [p],
 \end{equation}
 where we assume the relations $t_i^{m_i'}=0$ for $i\in [p]$.  Use (\ref{whitney}) to deduce that the top Chern number of
 $\tilde R(\m')$ is $c(\m',\bomega)$.

 From the results of \S\ref{sec:numsingtpl}, in particular Lemma \ref{globsecR}, we deduce that there exists a monomorphism $L_i:
 \C^{\m(\bomega)}\to \rH^0(\tilde R(i,\m'))$.  Furthermore $L_i(\C^{\m(\bomega)})$ generates $\tilde R(i,\m')$.
 Let $L=(L_1,\ldots,L_p):\oplus_{i\in[p]}\C^{\m(\omega)}\to \rH^0(\tilde R(\m'))$.  Then $L(\oplus_{i\in[p]}\C^{\m(\bomega)})$
 generates $\rH^0(\tilde R(\m'))$. Let $\delta:  \C^{\m(\bomega)}\oplus^p  \C^{\m(\bomega)}$ be the diagonal map.
 {We claim that $L\circ \delta$ almost generates $\rH^0(\tilde R(\m'))$.
  
 First, we consider a special case of Lemma \ref{condexistT} for $\cT\in \rS^{\bomega}$.     
Here we assume that $\x_1,\ldots,\x_d$ and $\y_1,\ldots,\y_d$ satsify the conditions induced by the equalities \eqref{bomegsymcon}:
\begin{eqnarray*}
\x_1=\ldots=\x_{\omega_1}(=\z_1),\ldots, \x_{d-\omega_p+1}=\ldots=\x_d=(\z_p), \\
\y_1=\ldots=\y_{\omega_1}(=\w_1),\ldots, \y_{d-\omega_p+1}=\ldots=\y_d=(\w_p).
\end{eqnarray*}
Then all parts of the lemma needed to be stated in terms of $\z_1,\ldots,\z_p$ and $\w_1,\ldots,\w_p$,
Second,  we restate Lemma \ref{globsecR} for $\cT\in \rS^{\bomega}$ and  $\x_1,\ldots,\x_d$ and $\y_1,\ldots,\y_d$ of the above form.
Third, let $Y_\alpha\subsetneq \Pi(\m')$,  where $\alpha$ are nonempty subsets of $[p]$,  be the varieties  defined in the proof of Theorem \ref{numbsingtup}.  
The proof of Theorem \ref{numbsingtup} yields that $L\circ\delta(\rS^{\bomega})$ almost generates $\tilde R(\m')$ with respect to the varieites $Y_\alpha$.
Theorem \ref{berttheom} yields that a generic $\cT\in \rS^{\bomega}$  has exactly $c(\m',\bomega)$ simple $\bomega$-symmetric singular vector tuples. 
The proof that a generic $\cT\in \rS^{\bomega}$ does not have a zero singular value is analogous to the proof given in Theorem \ref{numbsingtup}.}  \qed
 \begin{rem}\label{FOCScases} In the special case
 $\bomega=(1,1,\ldots,1)$ we have
 $c(\m',\bomega))=c(\m')$
 and Theorem \ref{numbsingtuppsym} reduces to Theorem \ref{numbsingtup}.
 In the case $\bomega=(d)$
 we have $c(m,\bomega)=\frac{(d-1)^{m}-1}{d-2}$ and Theorem \ref{numbsingtuppsym}
 reduces to the results in \cite{CS}.
 This last reduction was performed already in \cite{OO}.
 \end{rem}

 \begin{lemma}\label{2partition}
 In the case $\bomega=(d-1,1)$
 we have
 $$c((m_1,m_2),(d-1,1))=\sum_{i=0}^{m_1-1}\sum_{j=0}^{m_2-1}{i\choose j}
 (d-2)^j(d-1)^{i-j}.$$

 If $m_1\le m_2$ we have $c((m_1,m_2),(d-1,1))=\frac{(2d-3)^{m_1}-1}{2d-4}$

 If $m_1=m_2+1$ we have $c((m_1,m_2),(d-1,1))=\frac{(2d-3)^{m_1}-1}{2d-4}-(d-1)^{m_1-1}$
 \end{lemma}
 {We now compare our formulas for the $3\times 3\times 3$ partially symmetric tensors.
Consider first the case $c((3),(3))= \frac{2^3-1}{2-1}=7$, i.e. the Cartwright-Sturmfels formula.
That is, a generic symmetric $3\times 3\times 3$ tensor  has $7$ singular vector triples of the form  $([\x],[\x],[\x])$.
Second, consider a generic $(2,1)$ partially symmetric tensor.  
The previous lemma gives $c((3,3),(2,1))=13$.  I.e. a generic partially symmetric tensor  has $13$ the singular vector triples
 of the form $([\x],[\x],[\y])$.  Third, consider a generic $3\times 3\times 3$ tensor.  In this case
case our formula gives $c(3,3,3)=c((3,3,3), (1,1,1))=37$ singular vector triples of the form $[\x],[\y],[\z]$.  

Let us assume that we have a generic symmetric $3 \times 3 \times 3$ tensor.
Let us estimate the total number of singular vector triples it may have, assuming that it behaves as a generic 
partially symmetric tensor and a nonsymmetric one.  First it has $7$ singular value tuples f the form $[\x],[\x],[\x]$.
Second, it has  $3\cdot 6=18$ singular vector triples of the form $[\x],[\y],[\z]$ where exactly two out of these three classes are the same.
Third, it has $12$ singular vector triples of the form $[\x],[\y],[\z]$ where all three classes are distinct.}
Note also that the number $37$ was computed,
 in a similar setting, in \cite{Mas}.

 {The above discussed situation indeed occurs}
 for the diagonal tensor $\cT=[\delta_{i_1,i_2}\delta_{i_2i_3}]\in \C^{3\times 3\times 3}$.

 \begin{table}[h]
 \centering
 $$\begin{array}{c|c|c}
 (x0,x1,x2)(y0,y1,y2)(z0,z1,z2)& \textrm{singular value}\\
 \hline
  (1,0,0)(1,0,0)(1,0,0) &1\\
  (0,1,0)(0,1,0)(0,1,0) &1\\
  (0,0,1)(0,0,1)(0,0,1) &1\\
  (1,1,0)(1,1,0)(1,1,0) &1\\
  (1,0,1)(1,0,1)(1,0,1) &1\\
  (0,1,1)(0,1,1)(0,1,1) &1\\
  (1,1,1)(1,1,1)(1,1,1) &1\\
 \hline
   (1,1,0)(1,-1,0)(1,-1,0)  &1&\textrm{3 permutations}\\
  (1,0,1)(1,0,-1)(1,0,-1)  &1&\textrm{3 permutations}\\
  (0,1,1)(0,1,-1)(0,1,-1)  &1&\textrm{3 permutations}\\
  (1,1,1)(1,1-1)(1,1,-1)  &1&\textrm{3 permutations}\\
  (1,1,1)(1,-1,1)(1,-1,1)  &1&\textrm{3 permutations}\\
  (1,1,1)(-1,1,1)(-1,1,1)  &1&\textrm{3 permutations}\\
 \hline
 (1,0,0)(0,1,0)(0,0,1)  &0&\textrm{6 permutations}\\
  (1,1,-1)(1,-1,1)(-1,1,1)  &-1&\textrm{6 permutations}\\
 \end{array}$$
 \caption{List of singular value tuples of a  $3\times 3 \times 3$ diagonal tensor}
\label{fig2}
\end{table}

 In this list the first $7$ singular vectors have equal entries
 and they  are the one counted by the formula in \cite{CS}.
 The first $7+6=13$ singular vectors have the form $([\x],[\x],[\y])$.
 Any singular vector of this form gives $3$ singular vectors
 $([\x],[\x],[\y])$, $([\x],[\y],[\x])$, $([\y],[\x],[\x])$.
 Note that six singular vectors have zero singular value,
 but this does not correspond to the generic case,
 indeed for a generic tensor all $37$ singular vectors correspond to nonzero singular value.

 In the case of $4\times 4\times 4$ tensors,
 the diagonal tensor has $156$ singular vectors corresponding to nonzero singular value and infinitely many singular vectors corresponding
 to zero singular vectors. These infinitely many singular vectors
 fill exactly $36$ projective lines in the Segre product
 $\P(\C^4)\times\P(\C^4)\times\P(\C^4)$ which ``count''
 in this case for the remaining $240-156=84$ singular vectors.

 \section{A homogeneous pencil eigenvalue problem}
 For $\x=(x_1,\ldots,x_m)\trans\in \C^m$ denote $\x^{\circ (d-1)}:=(x_1^{d-1},\ldots,x_m^{d-1})\trans$.
 Let $\cT\in \C^{m^{\times d}}$.
 The eigenvalues of $\cT$ satisfying {\eqref{nleigprb}} are called the $E$-eigenvalues in \cite{Qi07}.
 The homogeneous eigenvalue problem introduced in \cite{LS47}, \cite{Lim05} and \cite{Qi05}, sometimes referred as $N$-eigenvalues,
 is
 \begin{equation}\label{homeigprb}
 \cT\times\otimes^{d-1}\x=\lambda\x^{\circ (d-1)}, \quad \x\ne \0.
 \end{equation}
 Let $\cS\in \C^{m^{\times d}}$.  Then a generalized $d-1$ \emph{pencil} eigenvalue problem is
 \begin{equation}\label{d-1penprob}
 \cT\times\otimes^{d-1}\x=\lambda\cS\times\otimes^{d-1}\x.
 \end{equation}
 For $d=2$ the above homogeneous system is the standard eigenvalue problem for a pencil of matrices $\cT-\lambda\cS$.

 A tensor $\cS$ is called \emph{singular} if the system
 \begin{equation}\label{shomeq}
 \cS\times\otimes^{d-1}\x=\0
 \end{equation}
 has a nontrivial solution.  Otherwise $\cS$ is called nonsingular.  It is very easy to give an example of a symmetric
 nonsingular $\cS$ \cite{Fri11}.  Let $\w_1,\ldots,\w_m$ be linearly independent in $\C^m$.  Then $\cS=\sum_{i=1}^m \otimes^d\w_i$
 is nonsingular.  The set of singular tensors in $\C^{m^{\times d}}$ is a given by the zero set of some multidimensional resultant
 \cite[Chapter 13]{GKZ}.  It can be obtained by elimination of variables.  {Let us denote by $\Res_{m,d}\in\C[\C^{m^{\times d}}]$} the multidimensional
 resultant corresponding to the system \eqref{shomeq}, which is a homogeneous polynomial in the entries of $\cS$ of degree
 $\mu(m,d)=m(d-1)^{m-1}$, see formula (2.12) of \cite[Chapter 9]{GKZ}.   Denote by $Z(\Res_{m,d})$  the zero set of the polynomial $\Res_{m,d}$. 
 Then $\Res_{m,d}$ is an irreducible polynomial such that the system \eqref{shomeq} has a nonzero solution
 if and only if $\Res_{m,d}(\cS)=0$.   Furthermore, for a generic point $\cS\in Z(\Res_{m,d})$
 the system \eqref{shomeq} has exactly one simple solution in $\P(\C^m)$.
  The eigenvalue problem \eqref{d-1penprob} consists of two steps.  First find all $\lambda$ satisfying $\Res_{m,d}(\lambda\cS-\cT)=0$.
 Clearly $\Res_{m,d}(\lambda\cS-\cT)$ is a polynomial in $\lambda$ of degree {at most $\mu(m,d)$}.  (It is possible that this polynomial in $\lambda$
 is a zero polynomial.  This is the case where there exists a nontrivial solution to the system $\cS\otimes^{d-1}\x=\cT\otimes^{d-1}\x=\0$.)  
After then one needs to find the nonzero solutions of the system $(\lambda\cS-\cT)\otimes^{d-1}\x=0$, which are viewed as eigenvectors in $\P(\C^m)$.
 Assume that $\cS$ is nonsingular.  Then $\Res_{m,d}(\lambda\cS-\cT)=\Res_{m,d}(\cS)\lambda^{\mu(m,d)}+$ polynomial in $\lambda$ 
of degree {at most $\mu(m,d)-1$}.
 We show below a result known to the experts, that for generic $\cS,\cT$ each eigenvalue $\lambda$ the system $(\lambda\cS-\cT)\otimes^{d-1}\x=0$ has exactly
 one corresponding eigenvector in $\P(\C^m)$.   We outline a short proof of the following known theorem,
 which basically uses only the existence of the resultant for the system \eqref{shomeq}.  {For an identity tensor $\cS$. i.e. \eqref{homeigprb}, see \cite{Qi05}.}
 \begin{theo}\label{neighompen}  Let $\cS,\cT\in\C^{m^{\times d}}$ and assume that $\cS$ is nonsingular.  
Then $\Res_{m,d}(\lambda\cS-\cT)$ is a polynomial in $\lambda$ of degree $m(d-1)^{m-1}$.  For a generic $\cS$ and $\cT$ 
to each eigenvalue $\lambda$ of the pencil \eqref{homeigprb}
 corresponds one eigenvector in $\P(\C^m)$.
 \end{theo}
 \proof  Consider the space $\P(\C^2)\times \P(\C^{m^{\times d}}\times\C^{m^{\times d}})\times \P(\C^m)$ with the local coordinates $((u,v),(\cS,\cT),\x)$.
 Consider the system of $m$-equation homogenous in $(u,v), (\cS,\cT), \x$ given by
 \begin{equation}\label{homsyseqST}
 (u\cS-v\cT)\times\otimes^{d-1}\x=\0.
 \end{equation}
 The existence of the multidimensional resultant is equivalent to the assumption that the above variety $V(m,d)$  is an irreducible variety
 of dimension $2m^d-1$ in $\P(\C^2)\times \P(\C^{m^{\times d}}\times \C^{m^{\times d}})\times \P(\C^m)$.  So it is enough to find a good 
point $(\cS_0,\cT_0)$ such that it has exactly $\mu(m,d)=m(d-1)^{m-1}$ smooth point {s} $((u_i,v_i), (\cS_0,\cT_0), \x_i)$ in $V(m,d)$.

 {We call $\cT=[t_{i_1,\ldots,i_d}]\in \C^{m^{\times d}}$ an almost diagonal tensor if $t_{i_1,\ldots,i_d}=0$ whenever 
$i_p\ne i_q$ for some $1<p<q\le d$.
 An almost diagonal tensor $\cT$ is represented by a matrix $B=[b_{ij}]\in \C^{m\times m}$ where $t_{i,j,\ldots,j}=b_{ij}$.} 
 Assume now that $\cS_0,\cT_0$ are {almost diagonal tensors represented by the matrices $A,B$ respectively.  Then
 \begin{equation}\label{defST0}
 \cS_0\times\otimes^{d-1}\x=A\x^{\circ (d-1)}, \; \cT_0\times\otimes^{d-1}\x=B\x^{\circ (d-1)}.
 \end{equation}
 Assume furthermore that $A=I$ and $B$ is a cyclic permutation matrix, i.e. $B(x_1,\ldots,x_m)\trans$ $=(x_2,\ldots,x_m,x_1)\trans$.}
 Then $B$ has $m$ distinct eigenvalues, the $m-th$ roots of unity.  $\x$ is an eigenvector of \eqref{defST0} if and only if $\x^{\circ (d-1)}$ is 
 an eigenvector of $B$.   Fix an eigenvalue of $B$.  One can fix $x_1=1$.
 Then we have exactly $(d-1)^{m-1}$ eigenvectors in $\P(\C^m)$ corresponding to each eigenvalue $\lambda$ of $B$.  So altogether we have $m(d-1)^{m-1}$
 distinct eigenvectors.  It is left to show that each point $((u_i,v_i),(\cS_0,\cT_0),\x_i)$ is a simple point of $V(m,d)$.
 For that we need to show that the Jacobian of the system \eqref{homsyseqST} at each point has rank $m$, the maximal possible rank, 
at $((u_i,v_i),(\cS_0,\cT_0),\x_i)$.  For that we assume that $u_i=\lambda_i, v_i=1, x_1=1$.  This easily follows from the fact that
 each eigenvalue of $B$ is a simple eigenvalue.  Hence the projection of $V(m,d)$ on $\P(\C^{m^{\times d}}\times \C^{m^{\times d}})$
 is $m(d-1)^{m-1}$ valued.

 Note that in this example each eigenvalue $\lambda$ of \eqref{defST0} is of multiplicity $(d-1)^{m-1}$.  It is left to show that when we consider the
 pairs $\cS_0,\cT$ where $\cT$ varies in the neighborhood of $\cT_0$ we obtain $m(d-1)^{m-1}$ different eigenvalues.  Since the Jacobian of the system
 \eqref{defST0} has rank $m$ at each eigenvalue $\lambda_i=\frac{u_i}{v_i}$  and the corresponding eigenvector $\x_i$, one has a simple variation formula for
 each $\delta\lambda_i$ using the implicit function theorem.  Fix $x_1=1$ and denote 
$\bF(\x,\lambda,\cT)=(F_1,\ldots,F_m):=(\lambda \cS_0-\cT)\times \otimes^{d-1}\x$.
 Thus we have the system of $m$ equations $\bF(\x,\lambda)=0$ in $m$ variables $x_2,\ldots,x_m, \lambda$.  We let $\cT=\cT_0+t\cT_1$ and we want to find
 the first term of $\lambda_i(t)=\lambda_i+\alpha_i t +O(t^2)$.  We also assume that $\x_i(t)=\x_i+t\y_i+O(t^2)$, where $\y_i=(0,y_{2,i},\ldots,y_{m,i})\trans$.
 {Let
 \[\z_i=\sum_{j\in [d-1]}\otimes^{j-1}\x_i\otimes \y_i \otimes^{d-1-j}\x_i.\]
 The} first order computations yields the equation
 \begin{equation}\label{fstordper}
 \cT_1\times \otimes^{d-1}\x_i +\cT_0\times\z_i=\alpha_i \cS_0\times\otimes^{d-1}\x_i+\lambda_i\cS\times\z_i.
 \end{equation}

 Let $\w=(w_1,\ldots,w_m)\trans$ be the left eigenvector of $B$ corresponding to $\lambda_i$, i.e. $\w\trans B=\lambda_i\w\trans$ normalized by the condition
 $\w\trans (\cS_0\times\otimes^{d-1}\x_i)=(\cS_0\times\otimes^{d-1}\x_i)\times\w=1$.  Contracting both sides of \eqref{fstordper} with the vector $\w$
 we obtain
 \begin{equation}
 \alpha_i=\cT_1\times (\w\otimes(\otimes^{d-1}\x_i)).
 \end{equation}

 It is straightforward to show that $\alpha_1,\ldots,\alpha_{m(d-1)^{m-1}}$ are pairwise distinct for a generic $\cT_1$ .
 \qed

 The proof of Theorem \ref{neighompen} yields {the following}.
 \begin{corol}\label{nNeig}  Let $\cT\in \C^{m^{\times d}}$ be a generic tensor.  Then the homogeneous eigenvalue problem
 \eqref{homeigprb} has exactly $m (d-1)^{m-1}$ distinct eigenvectors in $\P(\C^m)$, which correspond to distinct eigenvalues.
 \end{corol}

 We close this section with an heuristic argument which shows that a generic pencil $(\cS,\cT)\in \P(\C^{m^{\times d}}\times \C^{m^{\times d}})$
 has $\mu(m,d)=m(d-1)^{m-1}$ distinct eigenvalues in $\P(\C^m)$.  Let $\cS\in \C^{m^{\times d}}$ be nonsingular
 Then $\cS$ induces a linear map $\hat \cS$ {from} the line bundle $\otimes^{d-1} T(m)$ to the trivial bundle $\C^m$ over $\P(\C^m)$
 by $\otimes^{d-1}\x\mapsto \cS\times \otimes^{d-1}\x$.  Then we have an exact sequence of line bundles
 \[0\to \otimes^{d-1} T(m)\to \C^m \to Q_{m,d}\to 0\]
 where $Q_{m,d}=\C^m/(\hat\cS( \otimes ^{d-1} T(m)))$.  The Chern polynomial of $Q_{m,d}$ is $1+\sum_{i=1}^{m-1} (d-1)^i t^i \alpha^i$.
 A similar computation for finding the number of eigenvectors of \eqref{nleigprb} shows that the number of eigenvalues of \eqref{d-1penprob}
 is the coefficient of $t_1^{m-1}$ in the polynomial $\frac{\hat t_1^{m} -\tilde t_1^{m}}{\hat t_1-\tilde t_1}$.
 Here $\hat t_1=\tilde t_1=(d-1)t_1$.  Hence the coefficient of $t_1^{m-1}$ is $\frac{(d-1)^m-(d-1)^m}{(d-1)-(d-1)}$.
 The calculus interpretation of this formula is the derivative of $t^m$ at $t=d-1$, which gives the value of the coefficient $m(d-1)^{m-1}$.
  \section{On uniqueness of a best approximation}\label{sec:uniqbstap}
 Let $\an{\cdot,\cdot}, \|\cdot\|$ be the standard inner product and the corresponding Euclidean norm on $\R^n$.
 {For a subspace $\U\subset\R^n$ we denote by $\U^\perp$ the subspace of all orthogonal vectors to $\U$ in $\R^n$.}
 Let $C\subsetneq\R^n$ be a given nonempty closed set, {(in the Euclidean topology,  see \S1)}.  For each $\x\in \R^n$ we consider the function
  {
 \begin{equation}\label{defdistC}
 \dist(\x,C):=\inf \{\|\x-\y\|, \;\y\in C\} \;(\ge 0).
 \end{equation}
 We first recall that this infimum is achieved for at least one point $\y^\star\in C$,
 which is called a best approximation of $\x$.  Observe that $\|\x-\y\|\ge \|\y\|-\|\x\|$.  Hence, in the infimum \eqref{defdistC} it is enough to restrict
 the values of $\y$ to the 
 the compact set $C(\x):=\{\y\in C,\;\|\y\|\le \|\x\|+\dist(\x,C)\}$.  Since $\|\x-\y\|$ is a continuous function on $C(\x)$, it achieves its minimum
 at some point $\y^\star$, which will be sometimes denoted by $\y(\x)$.}
 
 The following result is {probably} well known and we bring its short proof for completeness.
 \begin{lemma}\label{bestnraprxlemma}  Let $C\subsetneq\R^n$ be a given closed set.
Let $\U\subset \R^n$ be a subspace with $\dim \U\in [n]$ and such that $\U$ is not contained in $C$.
Let $d(\x), \x\in \U$ be the restriction of $\dist(\cdot,C)$
 to $\U$.
\begin{enumerate}
\item The function $\dist(\cdot,C)$ is Lipschitz with
 constant  constant $1$:
 \begin{equation}\label{Lipscon}
 |\dist(\x,C)-\dist(\z,C)|\le \|\x-\z\| \textrm{ for all } \x,\z\in \R^n.
 \end{equation}
 \item The function $d(\cdot)$ is differentiable a.e. in $\U$.
\item Let $\x\in\U\setminus C$ and assume that $d(\cdot)$ is differentiable at $\x$.
 Denote the differential as $\partial d(\x)$, which is viewed as {a} linear functional
 on $\U$. Let $\y^\star \in C $ be a best approximation to $\x$.  Then
 \begin{equation}\label{forderdistV}
 \partial d(\x)(\uu)=\an{\uu,\frac{1}{\dist(\x,C)} (\x-\y^\star)} \textrm{ for each } \uu\in\U.
 \end{equation}
 If $\z^{\star}$ is another best approximation to $\x$ then {$\z^\star-\y^\star\in\U^\perp$}.
\end{enumerate}
\end{lemma}
\proof  Assuming that $\dist(\x,C)=\|\x-\y^\star\|$ we deduce the following inequality.
 \begin{equation}\label{basindist}
 \dist(\z,C)\le \|\z-\y^\star\| \textrm{ for each } \z\in \R^n.
 \end{equation}
 Suppose next that $\dist(\z,C)=\|\z-\y\|, \y\in C$.  Hence
 \[-\|\x-\z\|\le \|\x-\y^\star\|-\|\z-\y^\star\|\le \dist(\x,C)-\dist(\z,C)\le \|\x-\y\|-\|\z-\y\|\le \|\x-\z\|.\]
 This proves \eqref{Lipscon} and part \emph{1}.
 Clearly, $d(\cdot)$ is also Lipschitz on $\U$.  Rademacher's theorem yields that $d(\cdot)$ is differentiable a.e.,
 which proves part \emph{2}.
 To prove part \emph{3} we fix $\uu\in \U$.  Then
 \[\dist(\x+t\uu,C)=\dist(\x,C)+t\partial d(\x)(\uu)+to(t).\]
 \eqref{basindist} yields the inequality
 \[\dist(\x+t\uu,C)\le \|\x+t\uu-\y^\star\|=\|\x-\y^*\|+t\an{\uu,\frac{1}{\dist(\x,C)}(\x-\y^\star)} +O(t^2).\]
 Compare this inequality with the previous equality to deduce that
 \[t\partial d(\x)(\uu)\le t\an{\uu,\frac{1}{\dist(\x,C)}(\x-\y^\star)}\]
 for all $t\in\R$.  This implies \eqref{forderdistV}.  If $\z^{\star}$ another best approximation to $\x$ then
 \eqref{forderdistV} yields that $\z^\star-\y^\star\in\U^\perp$.\qed
 \begin{corol}\label{bestnraprxcoro}
 Let $C\subsetneq\R^n$ be a given closed set.
 \begin{enumerate}
 \item
 The function $\dist(\x,C)$ {is differentiable} a.e. in $\R^n$.
 \item Let $\x\in \R^n\setminus C$ and assume that $\dist(\cdot,C)$ is differentiable
 at $\x$.  Then $\x$ has a unique best approximation $\y(\x)\in C$.  Furthermore
 \begin{equation}\label{forderdist}
 \partial \dist(\x,C)(\uu)=\an{\uu,\frac{1}{\dist(\x,C)}(\x-\y(\x))} \textrm{ for each } \uu\in\R^n.
 \end{equation}
 {In particular, almost all $\x\in\R^n$  have a unique best approximation $\y(\x)\in C$.}
 \end{enumerate}

 \end{corol}
 \proof
 Choose $\U=\R^n$, so $d(\cdot)=\dist(\cdot,C)$ is differentiable a.e. by part \emph{2} of Lemma \ref{bestnraprxlemma}.
 This establishes part \emph{1} of our lemma.
 Assume that $\y^\star$ and $\z^\star$ are best approximations of $\x$.  Then {$\z^\star-\y^\star\in(\R^n)^\perp$  
 by part \emph{3} of Lemma \ref{bestnraprxlemma}.  As $(\R^n)^\perp=\{\0\}$ we obtain that $\z^\star=\y^\star$.} 
 Furthermore  \eqref{forderdist} holds.  \qed

 \section{Best rank one approximations of $d$-mode tensors}\label{sec:basfacts}
 On $\C^\m$ define an inner product and its corresponding Hilbert-Schmidt norm $\an{\cT,\cS}:=\cT\times \bar\cS, \|\cT\|=\sqrt{\an{\cT,\cT}}$.
 We first present some known results of best rank one approximations of  real tensors.
 In this section we assume that $\F=\R$ and $\cT\in\R^\m$.
 Let $\rS^{m-1}\subset\R^m$ be the $m-1$-dimensional sphere $\|\x\|=1$.
 Denote by $\rS(\m)$ the $d$-product of the spheres $\rS^{m_1-1}\times\ldots\times\rS^{m_d-1}$.  Let $(\x_1,\ldots,\x_d)\in\rS(\m)$  
and associate with $(\x_1,\ldots,\x_d)$ the $d$ one dimensional subspaces $\U_i=\span(\x_i)$, $i\in [d]$.
 Note that
 \[
 \|\otimes_{i\in[d]}\x_i\|=\prod_{i\in[d]}\|\x_i\|=1.
 \]
 The projection $P_{\otimes_{i\in[d]}\U_i}(\cT)$  of $\cT$ onto the one dimensional subspace 
$\U:=\otimes_{i\in[d]}\U_i \subset \otimes_{i\in[d]}\R^{m_i}$,
  is given by
 \begin{equation}\label{projform}
 f_{\cT}(\x_1,\ldots,\x_d)\otimes_{i\in[d]}\x_i, f_{\cT}(\x_1,\ldots,\x_d):=\an{\cT,\otimes_{i\in[d]}\x_i},(\x_1,\ldots,\x_d)\in\rS(\m).
 \end{equation}
 Let $P_{(\otimes_{i\in[d]}\U_i)^\perp}(\cT)$ be the orthogonal projection of $\cT$ onto the orthogonal complement of $\otimes_{i\in[d]}\U_i$.
 The Pythagorean identity yields
 \begin{equation}\label{pythiden}
 \|\cT\|^2=\|P_{\otimes_{i\in[d]\U_i}}(\cT)\|^2 + \|P_{(\otimes_{i\in[d]}\U_i)^\perp}(\cT)\|^2.
 \end{equation}
 With this notation, {a} best rank one approximation of $\cT$ from $\rS(\m)$ is given by
 \[
 \min_{(\x_1,\ldots,\x_d)\in\rS(\m)}\min_{a\in\R} \|\cT-a\otimes_{i\in[d]}\x_i\|.
 \]
 Observing that
  \[
  \min_{a\in\R} \|\cT-a\otimes_{i\in[d]}\x_i\|=
  \|\cT-P_{\otimes_{i\in[d]\U_i}}(\cT)\|=
  \|P_{(\otimes_{i\in[d]}\U_i)^\perp}(\cT)\|,
  \]
 it follows that {a} best rank one approximation is obtained by the minimization of $\|P_{(\otimes_{i\in[d]}\U_i)^\perp}(\cT)\|$.
  In view of \eqref{pythiden} we deduce that best rank one approximation is obtained by  the maximization of $\|P_{\otimes_{i\in[d]\U_i}}(\cT)\|$ and
 finally, using \eqref{projform}, it follows that {a} best rank one approximation is given by
 \begin{equation}\label{specnorm}
 \sigma_1(\cT):=\max_{(\x_1,\ldots,\x_d)\in\rS(\m)} f_{\cT}(\x_1,\ldots,\x_d).
 \end{equation}
  As in the matrix case $\sigma_1(\cT)$ is called in \cite{HilL10} the \emph{spectral norm}.  
Furthermore it is shown in \cite{HilL10} that the computation  of $\sigma_1(\cT)$ in general is NP-hard for $d>2$.

 We will make use of the following result of \cite{Lim05}, where we present the proof for completeness.
 \begin{lemma}\label{critptsf}  For $\cT\in\R^{\m}$, the
 critical points of $f|_{\rS(\m)}$, defined in \eqref{projform}, are singular value tuples satisfying
 \begin{equation}\label{singvalvecd}
  \cT\times(\otimes_{j\in[d]\setminus\{i\}}\x_j)=\lambda \x_i\
   \textrm{ for all }i\in[d],\ (\x_1,\ldots,\x_d)\in\rS(\m).
 \end{equation}
 \end{lemma}
 \proof We need to find the critical points of $\an{\cT,\otimes_{j\in[d]}\x_j}$ where $(\x_1,\ldots,\x_d)\in \rS(\m)$.
 Using Lagrange multipliers we consider the auxiliary function
 \[
 g(\x_1,\ldots,\x_d):=\an{\cT,\otimes_{j\in[d]}\x_j}-\sum_{j\in[d]}\lambda_j \x_j\trans\x_j.
 \]
 The critical points of $g$ then satisfy
 \[
 \cT\times(\otimes_{j\in[d]\setminus\{i\}}\x_j)=\lambda_i \x_i, \quad i\in [d],
 \]
 and hence
 $\an{\cT,\otimes_{j\in[d]}\x_j}=\lambda_i\x_i\trans\x_i=\lambda_i$ for all $i\in[d]$, which implies \eqref{singvalvecd}.
 \qed

 Observe next that $(\x_1,\ldots,\x_d)$ satisfies  \eqref{singvalvecd} if and only if the vectors $(\pm\x_1,\ldots,\pm\x_d)$ satisfy \eqref{singvalvecd}.
 In particular, we could choose the signs in $(\pm\x_1,\ldots,\pm\x_d)$ such that each corresponding $\lambda$ is nonnegative and then these
 $\lambda$ can be interpreted  as the singular values of $\cT$.  The maximal singular value of $\cT$ is denoted by $\sigma_1(\cT)$ and is given
 by \eqref{specnorm}.  Note that to each nonnegative singular value there are at least $2^{d-1}$ singular vector tuples of the form $(\pm\x_1,\ldots,\pm\x_d)$.
 So it is more natural to view the singular vector tuples $(\x_1,\ldots,\x_d)$ as points {$([\x_1],\ldots,[\x_d])$} 
in the real projective Segre product $\Pi_{\R}(\m)$.  Furthermore, the projection of $\cT$ on the one 
dimensional subspace spanned by $\otimes_{i\in [d]}(\pm \x_i)$, where
 $(\x_1,\ldots,\x_d)\in\rS(\m)$, is equal to one vector $(\cT\times\otimes_{i\in [d]}\x_i)\otimes_{i\in [d]} \x_i$.
 {
 \begin{theo}\label{fntbr1ap}  For almost all 
 $\cT\in \R^{\m}$ a best rank one approximation is unique.
 \end{theo}
 \proof    Let 
 \begin{equation}\label{defC(m)}
 C(\m):=\{\cT\in\R^\m, \;\cT=\otimes_{j\in[d]} \x_j,\;\x_j\in\R^{m_j},\;j\in[d]\}.
 \end{equation}
 $C(\m)$ is a compact set consisting of rank one tensors and the zero tensor.
 Corollary \ref{bestnraprxcoro} yields that for almost all $\cT$ {a} best rank one approximation is unique.  \qed
 }

 Note that Theorem \ref{fntbr1ap} implies part \emph{1} of Theorem \ref{fntbr1ap+s}.
 Let $\bomega=(\omega_1,\ldots,\omega_p)$ be a partition of $d$.  For $\cT\in \rS^{\bomega}(\R)$ it is natural to consider a 
best rank one approximation to $\cT$ of the form $\pm \prod_{i\in [p]} \otimes^{\omega_i}\x_i$ where {$\x_i\in \R^{m_i'},i\in [p]$}.
 We call such an approximation a best $\bomega$-symmetric rank one approximation.
 (The factor $\pm $ is needed only if each $\omega_i$ is even.)   As in the case $\cT\in \R^{\m}$ a best $\bomega$-symmetric rank one approximation
 of $\cT\in \rS^{\bomega}(\R)$ is a solution to the following maximum problem.
 \begin{equation}\label{omsymbrnk1ap}
 \max_{(\x_1,\ldots,\x_p)\in\rS(\m')} |\cT\times\otimes_{i\in[p]}\otimes^{\omega_i}\x_i|.
 \end{equation}
 As before, the critical points of the functions $\pm\cT\times\otimes_{i\in[p]}\otimes^{\omega_i}\x_i$ on $\rS(\m')$ satisfy
 \begin{equation}\label{critptsTsym}
 \cT\times \otimes_{j\in[p]}\otimes^{\omega_j-\delta_{ji}}\x_j =\lambda\x_i, \quad i\in [p], \quad (\x_1,\ldots,\x_p)\in \rS(\m').
 \end{equation}
 A best $\bomega$-symmetric rank one approximation corresponds to all $\lambda$ for which $|\lambda|$ has a maximal possible value.
 The arguments of the proof of Theorem \ref{fntbr1ap}  imply the following result.
 \begin{prop}\label{finapbrnk1symt}  {For almost all
 $\cT\in \rS^{\bomega}(\R)$} a best rank one $\bomega$-symmetric approximation is unique.
 \end{prop}

 Assume that {$\otimes_{j\in[d]} \y_j\in\R^{\m(\bomega)}$} is a best rank one approximation to a tensor $\cT\in \rS^{\bomega}(\R)$.
 It is not obvious a priori that $\otimes_{j\in[d]} \y_j$ is $\bomega$-symmetric. However, the following result is
 obvious.
 \begin{eqnarray}\label{sigbr1ap}
 &&\otimes_{j\in [d]}\y_{\sigma(j)} \textrm{ is best rank one approximation of } \cT\in \rS^{\bomega}(\R) \\
 &&\textrm{for each permutation } \sigma: [d]\to [d] \textrm{ which preserves }\rS^{\bomega}(\R).\notag
 \end{eqnarray}
  \begin{lemma}\label{semuniqbr1apprsym}  For a.a. $\cT\in \rS^{\bomega}(\R)$ there exists a unique rank one tensor 
{$\otimes_{j\in[d]}\y_j\in \R^{\m(\bomega)}$} such that all best rank one approximations of $\cT$ are of the form \eqref{sigbr1ap}.
 \end{lemma}
 
 To prove this lemma we need an auxiliary lemma.

 \begin{lemma}\label{permutlem}  Let $\otimes_{j\in[d]}\x_j,\otimes_{j\in[d]}\y_j\in \R^{n^{\times d}}$.  Assume that
 \begin{equation}\label{permutlem1}
 \an{\otimes_{j\in d}\x_j,\otimes^d\uu}=\an{\otimes_{j\in d}\y_j,\otimes^d\uu}\;\forall \uu\in \R^n.
 \end{equation}
 Then there exists a permutation $\sigma: [d]\to [d]$ such that $\otimes_{j\in[d]}\y_j=\otimes_{j\in[d]}\x_{\sigma(j)}$.
 \end{lemma}
 \proof
 Note that the condition \eqref{permutlem1} is equivalent to the equality
 \begin{equation}\label{permutlem2}
 \prod_{j\in [d]}\uu\trans \x_j=\prod_{j\in [d]}\uu\trans \y_j \;\forall\uu\in\R^n.
 \end{equation}
 If $\otimes_{j\in d}\x_j=0$ then $\prod_{j\in [d]}\uu\trans \y_j=0$ for all $\uu$.  Hence $\y_j=\0$ for some $j$, so
 $\otimes_{j\in[d]}\y_j=\otimes_{j\in[d]}\x_j=0$.  So we assume that $\otimes_{j\in[d]}\x_j,\otimes_{j\in[d]}\y_j$ are both nonzero.

 We now prove the lemma by induction.  For $d=1$ the lemma is trivial.  Assume that the lemma holds for $d=k$.  Let $d=k+1$.
 Assume that $\uu\in \span(\x_{k+1})^\perp$.  Then \eqref{permutlem2} yields that  $\prod_{j\in [d]}\uu\trans \y_j=0$.
 Hence $\span(\x_{k+1})^\perp \subset \cup_{j\in[k+1]}\span(\y_j){^\perp}$.  
Therefore there exists $j\in [k+1]$ such that $\span(\x_{k+1})^\perp=
 \span(\y_{j})^\perp$.  So $\y_j=t\x_{k+1}$ for some $t\in \R\setminus \{0\}$.  Hence there exist $\z_1,\ldots,\z_{d+1}\in\R^n$ 
and a permutation $\sigma':[k+1]\to[k+1]$
 such that $\otimes_{j\in [k+1]} \z_{\sigma'(j)}=\otimes_{j\in [k+1]}\y_j$ where {$\z_{k+1}=\x_{k+1}$}.  
{Thus} $\otimes_{j\in[k+1]}\x_j$ and
 $\otimes_{j\in[k+1]}\z_j$ satisfy \eqref{permutlem2}.  Therefore  $\otimes_{j\in[k]}\x_j$ and $\otimes_{j\in[k]}\z_j$ satisfy
 \eqref{permutlem2}.  Use the induction hypothesis to deduce the lemma.  \qed

 \textbf{Proof of Lemma \ref{semuniqbr1apprsym}.}
 We use {part \emph{3} of} Lemma \ref{bestnraprxlemma} as follows.  Let {$\R^n=\R^{\m(\bomega)}$} and assume that {$C=C(\m(\bomega))$ as defined
 in \eqref{defC(m)}.}   We let $\U:=\rS^{\bomega}(\R)$.
 Assume that $d(\cdot)$ is differentiable at $\cT\in \rS^{\bomega}(\R)\setminus C$.
 Suppose that $\otimes_{j\in[d]} \y_j,\otimes_{j\in[d]} \z_j$ are best rank {one} approximations of $\cT$.
 So
 \[\sigma_1(\cT)=\| \otimes_{j\in[d]} \y_j\|=\prod_{j\in [d]} \|\y_j\|=\| \otimes_{j\in[d]} \z_j\|=\prod_{j\in [d]} \|\z_j\|>0.\]
 Without loss of generality we may assume that
 \begin{equation}\label{xznorm}
 \|\y_j\|=\|\z_j\|=\sigma_1(\cT)^{\frac{1}{d}}\; \forall j\in [d].
 \end{equation}
 {Lemma \ref{bestnraprxlemma}} yields that
 \[\an{\otimes_{i\in[p]} \otimes^{\omega_i} \uu_i,\otimes_{j\in[d]}\y_j-\otimes_{j\in[d]}\z_j}=0 \;\forall \uu_i\in\R^{m'_i}\; i\in [p].\]
 The above equality is equivalent to
 \begin{equation}\label{permparsymcon}
 \prod_{i\in[p]}\prod_{j_i\in[\omega_i]} \uu_i\trans\y_{\alpha_i+j_i}=\prod_{i\in[p]}\prod_{j_i\in[\omega_i]} \uu_i\trans\z_{\alpha_i+j_i},\;
 \;\forall \uu_i\in\R^{m_i'},\;i\in [p],
 \end{equation}
 where $\omega_0=0$ and $ \alpha_i=\sum_{k=0}^{i-1} \omega_k$ for all $i\in[p]$.

 Suppose first that $p=1$, i.e. $\rS^{\bomega}(\R)$ is the set of all symmetric tensors in $\R^{m_1^{\times\omega_1}}$.  (Note that $d=\omega_1$.)
 Then Lemma \ref{permutlem} and \eqref{permparsymcon} yields that $\otimes_{j\in[d]}  \z_j=\otimes_{j\in [d]} \y_{\sigma(j)}$ for some permutation $\sigma:[d]
 \to [d]$.  This proves our lemma for $p=1$.

 Assume now that $p>1$.  Fix $k\in [p]$.  Fix $\uu_i\in i\in [p]\setminus\{k\}$.  Let
 \[s_k:=\prod_{i\in [p]\setminus\{k\}}\prod_{l_j\in[\omega_j]}\uu_{l_j}\trans \y_{\alpha_{j}+l_j}, \quad
 t_k:=\prod_{i\in [p]\setminus\{k\}}\prod_{l_j\in[\omega_j]}\uu_{l_j}\trans \z_{\alpha_{j}+l_j}.\]
 Assume that $s_k\ne 0$.  Then the two rank one tensors $s_k\otimes_{l_k\in [\omega_k]}\y_{\alpha_k+l_k}, t_k\otimes_{l_k\in [\omega_k]}\z_{\alpha_k+l_k}
 \in \R^{(m_k')^{\times \omega_k}}$ satisfy the assumptions of Lemma \ref{permutlem}.  Hence there exists a permutation $\sigma_k: [\omega_k]\to [\omega_k]$
 such that $t_k\otimes_{l_k\in [\omega_k]}\z_{\alpha_k+l_k}=s_k\otimes_{l_k\in [\omega_k]}\y_{\alpha_k+\sigma_k(l_k)}$.
 In view of \eqref{xznorm} we deduce the equality $\otimes_{l_k\in [\omega_k]}\z_{\alpha_k+l_k}=\pm\otimes_{l_k\in [\omega_k]}\y_{\alpha_k+\sigma_k(l_k)}$.
 Hence there exists $\omega:[d]\to [d]$ which leaves invariant each set $[\alpha_{j+1}]$ for $j\in [p-1]$ such that
 $\otimes_{j\in [d]}\z_j=\pm \times_{j\in[d]}\y_{\sigma(j)}$.  As $\otimes_{j\in [d]}\z_j$ and $\otimes_{j\in [d]}\y_j$ are best rank one approximation
 to $\cT$ we deduce that $\otimes_{j\in [d]}\z_j= \otimes_{j\in[d]}\y_{\sigma(j)}$.  \qed

 {A recent result of the first author claims} that each $\cT\in\rS^{\bomega}(\R)$ has a best rank one approximation which is $\bomega$-symmetric
 \cite[Theorem 1]{Fri11}.  For symmetric tensors this theorem is {equivalent to the old theorem of Banach \cite{Ban38}.  (See \cite[Theorem 4.1]{CHLZ} for another
 proof of Banach's theorem.)   We now give a refined version of \cite[Theorem 1]{Fri11}, whose proof {uses} of the results in \cite{Fri11}.}

 \begin{theo}\label{unfinapbrnk1symt} Each $\cT\in\rS^{\bomega}(\R)$ has a best rank one approximation which is $\bomega$-symmetric.
 Furthermore, for almost all  $\cT\in \rS^{\bomega}(\R)$ {a} best rank one approximation is unique and $\bomega$-symmetric.
 \end{theo}
 \proof   {The claim that each $\cT\in\rS^{\bomega}(\R)$ has a best rank one approximation which is $\bomega$-symmetric is proved in \cite{Fri11}.
 It is left to show that  for a.a.  $\cT\in \rS^{\bomega}(\R)$ a best rank one approximation is unique and $\bomega$-symmetric.
 Lemma \ref{semuniqbr1apprsym} claims that for a.a. $\cT\in \rS^{\bomega}(\R)$ there exists a unique rank one tensor 
$\otimes_{j\in[d]}\y_j\in \R^{\m(\bomega)}$ such that all best rank one approximations of $\cT$ are of the form \eqref{sigbr1ap}.
The first part of the theorem yields that one of these best rank approximations $\otimes_{j\in[d]} \y_j\in\R^{\m(\bomega)}$ is 
$\bomega$-symmetric .  Hence all the tensors of the form \eqref{sigbr1ap} are equal to $\otimes_{j\in[d]} \y_j\in\R^{\m(\bomega)}$.\qed} 

Note that part \emph{2} of Theorem \ref{fntbr1ap+s} follows from Theorem \ref{unfinapbrnk1symt}.
 \section{Best rank-$\br$ approximation}

 {In the first part of this section we assume that $\F$ is any field.}
Let $\m=(m_1,\ldots,m_d)\in\N^d$, $M=\prod_{j\in d} m_j$,  $M_i=\frac{M}{m_i}$  and $\m_i=(m_1,\ldots,m_{i-1},m_{i+1},\ldots,m_d)\in \N^{d-1}$ for $i\in [d]$.
 Assume that $\cT=[t_{i_1,\ldots,i_d}]\in\F^\m$.
 Denote by $T_i\in \F^{m_i\times M_i}$ the unfolded matrix of the tensor
 $\cT$ in the mode $i$.  That is, let $\cT_{j,k} \in \F^{\m_k}$ be the following $d-1$ mode tensor.  Its entries are
 $[t_{i_1,\ldots,i_{k-1},j,i_{k+1},\ldots,i_d}]$ for $i_p\in [m_p], p\in [d]\setminus\{k\}$. So $j\in [m_k]$.
 Then the row $j$ of $T_i$ is a tensor $\cT_{j,i}$ viewed as a vector in $\F^{\m_i}$.  Then $\mathrm{rank}_i\cT$ is
 the rank of the matrix $T_i$. $T_i$ can be seen as the matrix of the contraction map 
$\otimes_{j\in[d]\setminus\{ i\}}(\F^{\vee})^{m_j}\to \F^{m_i}$
 for $i\in [d]$.
 Clearly,
 \begin{equation}\label{rankineq}
 \mathrm{rank}_i\cT\le \min(m_i,M_i) \quad i\in [d].
 \end{equation}

 Carlini and Kleppe characterized the possible $r_i$ occurring as in the following Theorem.
 \begin{theo}[\cite{CK11}, Theorem 7]\label{posrank}  Suppose that $r_i\in[m_i]$ for $i\in [d]$.   Then there exists $\cT\in \F^\m$
 such that $\mathrm{rank}_i\cT=r_i$ for $i\in[d]$ if and only if
 \begin{equation}\label{rcompcond}
 r_i^2\le \prod_{j\in[d]}r_j \quad \textrm{ for each } i\in [d].
 \end{equation}
 \end{theo}

 We show a related argument working {over} any infinite field.
 For each $i$ let $f_i$ be one minor of $T_i$ of order $\min(m_i,M_i)$ .
 Let $f=\prod_{i\in[d]} f_i$, which is a nonzero polynomial in the entries of $\cT=[t_{i_1,\ldots,i_d}]$.  Let $V(\m)\subset \F^\m$ be the zero set of $f$.
 \begin{theo}\label{rankthm}  Let $\m\in\N^d$ and assume that $V(\m)\subset \F^\m$ is defined as above.  Then for each $\cT\in \F^\m\setminus V(\m)$
 the following equality holds.
 \begin{equation}\label{rankformula}
 \mathrm{rank}_i\cT=\min(m_i,M_i) \textrm{ for } i\in [d].
 \end{equation}
 In particular for $\F$ being a infinite field, a generic tensor $\cT\in\F^\m$ satisfies  \eqref{rankformula}.
 \end{theo}
 \proof  Suppose first that $m_i\le M_i$.  We claim that the $m_i$ tensors $\cT_{1,i},\ldots,\cT_{m_i,i}$ are linearly independent.
 Suppose not.  Then any $m_i \times m_i$ minor of $T_i$ is zero.  This contradicts the assumption that $\cT\in \F^\m\setminus V(\m)$.
 Hence $\mathrm{rank}_i\cT=m_i$.  Suppose that $m_i>M_i$.  {Let  $\cT_{k_1,i},\ldots,
 \cT_{k_{M_i},i}$ be the $M_i$ tensors which contribute to the minor $f_i$.  Since $f_i(\cT)\ne 0$ we deduce that
 $\cT_{k_1,i},\ldots, \cT_{k_{M_i},i}$  are linearly independent.}  Hence $\mathrm{rank}_i\cT_i=M_i$ for each $i\in[d]$.
 Since $f$ is a nonzero polynomial, for an infinite field $\F$ $V(\m)$ is a proper closed
 subset of $\F^\m$ in the Zariski topology.  Hence
 \eqref{rankformula} holds  {for} a generic tensor.
 \qed

 {Over} infinite fields, Theorem \ref{posrank} can be proved as a consequence of Theorem \ref{rankthm}.  {Indeed, let $\br=(r_1,\ldots,r_d)\in\N^d$
 and assume that \eqref{rcompcond} holds.  Choose a generic $\cT'=[t_{j_1,\ldots,j_d}']\in \F^\br$.  So $\mathrm{rank}_i\cT'=r_i, i\in [d]$.
 Extend $\cT'$ to $\cT=[t_{i_1,\ldots,i_d}]\in\F^\m$ by adding zero entries.  I.e. $t_{j_1,\ldots,j_d}=t_{j_1,\ldots,j_d}'$ for $j_i\in [r_i], i\in [d]$,
 and all other entries of $\cT$ are zero.  Then $\mathrm{rank}_i\cT=r_i, i\in [d]$.}

 In what follows we assume that $\F=\R$.
 Observe that the set of tensors
 having rank-$(r_1,\ldots, r_d)$ contains in the closure exactly all tensors of
 rank-$(a_1,\ldots, a_d)$ with $a_i\le r_i$. This closure is an algebraic variety, {defined
 as the zero set of all the minors of order $r_i+1$ of $T_i$ for $i\in [d]$.  We denote it by $C_{\bf r}$.} Note that having
 rank $(1,\ldots , 1)$ is equivalent to have rank $1$.

 Clearly  $C_{\bf r}$ is a closed set in $\R^{\m}$.
 The best $\br$-rank approximation of $\cT$ is the closest tensor in $C_{\bf r}$ to $\cT$ in {the Hilbert-Schmidt norm} \cite{LMV00}.
 Corollary \ref{bestnraprxcoro} yields.
 \begin{theo}\label{bestrankbrapprox}  Let $\m=(m_1,\ldots,m_d),\br=(r_1,\ldots,r_d)$ where $r_i\in[m_i]$ for $i\in [d]$
{and they satisfy (\ref{rcompcond}).}
 Then almost all $\cT\in \R^{\m}$ have a unique best $\br$-rank approximation.
 \end{theo}

 Let $\bomega=(\omega_1,\ldots,\omega_p)$ be a partition of $d$, $\m'=(m_1',\ldots,m_p')$ and assume that  $\m(\bomega)$ is defined by \eqref{defmomega}.
 Assume that $\br'=(r_1',\ldots,r_p')$,  where $r_i'\in [m_i']$ for $i\in[p]$.
 Let $\br(\bomega)=(\underbrace{r_1',\ldots,r_1'}_{\omega_1},\ldots,\underbrace{r_p',\ldots,r_p'}_{\omega_p})$.

 Let $C'_{\bf r'}=C_{\br(\bomega)}\cap \rS^{\bomega}$  .
 Clearly, $C'_{\bf r'}$ is a closed set, consisting of $\omega$-symmetric tensors in $\R^{\m(\bomega)}$ having rank $\br(\bomega)$ .

 Let $\cT\in \rS^{\bomega}$.  Then a best $\omega$ symmetric $\br(\bomega)$-rank approximation of $\cT$ is the closest tensor in $C'_{\bf r'}$
 to $\cT$.  Corollary \ref{bestnraprxcoro} yields.

 \begin{theo}\label{bestrankbrapproxsym}  Let $\bomega=(\omega_1,\ldots,\omega_p)$ be a partition of $d$.  Assume that $\m'=(m_1',\ldots,m_p'),
 \br'=(r_1',\ldots,r_p'), r_i'\in [m_i'], i\in [p]$
{and that $\m(\bomega)$ satisfies (\ref{rcompcond}).}
 Then almost all $\cT\in \rS^{\bomega}$ have a unique best $\omega$-symmetric $\br(\bomega)$-rank approximation.
 \end{theo}

  We close our paper with the following problem.  Let {$\cT\in\rS^{\bomega}$}.  Does $\cT$ have a best $\br(\bomega)$-rank approximation which
 is $\omega$-symmetric?  If the answer is yes, is a best $\br(\bomega)$-rank approximation unique for almost all $\cT\in\rS^{\bomega}$?
 In the previous section we showed that for $\br(\omega)=(1,\ldots,1)$ the answers to these problems are yes.

 \bibliographystyle{plain}

\begin{thebibliography}{MMM}
 

 \bibitem{Ban38} S. Banach, \"Uber homogene Polynome in ($L^2$), \emph{Studia Math.} 7 (1938), 36--44.

 \bibitem{CK11} E. Carlini and J. Kleppe, Ranks derived from multilinear maps, \emph{Journal of Pure and Applied Algebra}, 215 (2011), 1999--2004.

 \bibitem{CS} D. Cartwright, B. Sturmfels, The number of eigenvectors of a tensor,  \emph{Linear Algebra Appl.} 438 (2013), no. 2, 942-–952.

 \bibitem{CHLZ}  B. Chen, S. He, Z. Li, and S, Zhang, Maximum block improvement and polynomial optimization, \emph{SIAM J. Optimization},  22 (2012), 87--107.

 \bibitem{Che46} S. S. Chern, Characteristic classes of Hermitian Manifolds, \emph{Annals of Mathematics},  47 (1946),  85–-121.

 \bibitem{LMV00} L. de Lathauwer, B. de Moor and J. Vandewalle, On the best rank--1 and rank--$(R_1,\ldots,R_N)$ approximation of higher-order tensors,
 \emph{SIAM J. Matrix Anal. Appl.} 21 (2000), 1324--1342.


 \bibitem{Fri11} S. Friedland. Best rank one approximation of real symmetric tensors can be chosen symmetric, \emph{Front. Math. China} 8 (2013),  19– 40.

 \bibitem{Ful} W. Fulton,\emph{Intersection Theory}, Springer, Berlin 1984


 \bibitem{GKZ} I.~M.~Gelfand, M.~M.~Kapranov, A.~V.~Zelevinsky.
 \newblock {\em Discriminants, Resultants and Multidimensional
 Determinants} \newblock Birkh\"auser, Boston, 1994.

 \bibitem{GolV96} G.H. Golub and  C.F. Van Loan. {\em Matrix
  Computations}. John Hopkins Univ. Press, Baltimore, Md, USA, 3rd Ed., 1996.

 \bibitem{GH78} {P. Griffiths and J. Harris, \emph{Principles of Algebraic Geometry}, Wiley 1978.}


 \bibitem{Har} R. Hartshorne, Algebraic Geometry, Graduate Texts in Mathematics 52, Springer, 1977, New York

 \bibitem{HilL10} C.J. Hillar and L.-H. Lim. Most tensor problems are NP hard, \emph{Journal of the ACM}, 2013, to appear.

 \bibitem{Hir} { F. Hirzebruch, {\it Topological Methods in Algebraic Geometry},
 \\Grundlehren der math. Wissenschaften, vol. 131, Springer, 1966.}

 \bibitem{Kob} {S. Kobayashi, \emph{Differential Geometry of Complex Vector Bundles}, Princeton University Press 1987.}

 \bibitem{Lim05} L.-H. Lim. Singular values and eigenvalues of tensors:
 a variational approach. \emph{Proc. IEEE International Workshop on
 Computational Advances in Multi-Sensor Adaptive
 Processing} (CAMSAP '05), 1 (2005), 129-132.

 \bibitem{LS47}L. Lyusternik, and L. Shnirel'man, Topological methods in variational problems and their application to the differential
 geometry of surfaces. (Russian) Uspehi Matem. Nauk (N.S.) 2, (1947). no. 1(17), 166--217.

 \bibitem{Mas} C. Massri, Algorithm to find a maximum of a multilinear map over a product of spheres,
 arXiv:1110.6217

 \bibitem{NQWW} G. Ni, L. Qi, F. Wang,Y. Wang, The degree of the
 $E$-characteristic polynomial of an even order tensor, \emph{J. Math. Anal. Appl.}
 329(2007), n.2, 1218-1229


 \bibitem{OO} L. Oeding, G. Ottaviani, Eigenvectors of tensors
 and algorithms for Waring decomposition, \emph{J. Symbolic Comput.} 54 (2013), 9–-35.

 \bibitem{Qi05} L. Qi, Eigenvalues of a real supersymmetric tensor, \emph{J. Symbolic Comput.} 40 (2005)
 1302-1324.

 \bibitem {Qi07} {L. Qi: Eigenvalues and invariants of tensors, 
\emph{J. Math. Anal. Appl.} 325 (2007) 1363--1377.}

 \bibitem{ZLQ} X. Zhang, C. Ling, L. Qi,    The best rank-1 approximation
 of a symmetric tensor and related spherical optimization problems, preprint 2012


  \end{thebibliography}
 
  \end{document}